\documentclass{amsart}

\usepackage{amsmath,amssymb,amsfonts,amsthm,amscd,amsxtra}
\usepackage{mathrsfs}
\usepackage{mathtools}
\usepackage{cases}
\usepackage{latexsym}
\usepackage{bm}
\usepackage{bbm}
\usepackage{dsfont}

\usepackage{indentfirst}
\usepackage{fancyhdr}
\usepackage{ifthen}
\usepackage{verbatim}
\usepackage{scalefnt}

\usepackage{graphicx}      % 推荐统一用它（替代 graphics / epsfig 的日常用途）
\usepackage{epic,eepic}    % 你画 picture 需要
\usepackage{pb-diagram}
\usepackage{youngtab}

\usepackage[arrow,matrix]{xy}

\usepackage{tikz}
\usetikzlibrary{knots}
\usepackage{tikz-cd}

\usepackage{extpfeil}

\usepackage[colorlinks=true,linkcolor=blue,citecolor=blue,urlcolor=blue]{hyperref}

\newcommand{\lam}{\lambda}

\makeatletter
\@namedef{subjclassname@2020}{\textup{2020} Mathematics Subject Classification}
\makeatother

\DeclareMathOperator{\End}{End} 
 \DeclareMathOperator{\Ker}{Ker}
 
\DeclareMathOperator{\id}{id}

\DeclareMathOperator{\Hom}{Hom} \DeclareMathOperator{\sign}{sign}
\DeclareMathOperator{\Ext}{Ext}

\DeclareMathOperator{\Image}{Im}

\numberwithin{equation}{section}

\newcommand{\mbb}{\mathfrak{B}}

\newcommand{\A}{\mathscr{A}}

\newcommand{\bBS}{\mathfrak S}

\newcommand{\mmU}{\mathbb U}

\newcommand{\mmQ}{\mathbb Q}

\theoremstyle{plain}
\newtheorem{theorem}{Theorem}[section]
\newtheorem{lemma}[theorem]{Lemma}
\newtheorem{proposition}[theorem]{Proposition}
\newtheorem{corollary}[theorem]{Corollary}

\theoremstyle{definition}
\newtheorem{definition}[theorem]{Definition}

\begin{document}

\title[Dual partially harmonic tensors]
{Dual partially harmonic tensors and quantized Schur--Weyl duality}

\author{Pei Wang,  Zhankui Xiao}

\address{Wang: Techers' College, Beijing Union University,
Beijing, 100011, P. R. China}

\email{ldwangpei@buu.edu.cn}

\address{Xiao: School of Mathematical Sciences, Huaqiao University,
Quanzhou, Fujian, 362021, P. R. China}

\email{zhkxiao@hqu.edu.cn}

\thanks{Xiao is supported by the Natural Science Foundation of Fujian Province
(No. 2023J01126).\\}


\begin{abstract}
Let $V$ be a $2m$-dimensional symplectic space over an infinite field $K$. Let $\mathfrak{B}^{(f)}_{n,K}$
be the two-sided ideal of the Birman--Murakami--Wenzl algebra $\mathfrak{B}_{n,K}$ generated by
$E_1E_3\cdots E_{2f-1}$ with $1\leq f\leq\left\lfloor \frac n2 \right\rfloor$. In this paper, using the diagram category of framed tangles
and canonical basis, we prove that the natural homomorphism from $\mathfrak{B}_{n,K}/\mathfrak{B}^{(f)}_{n,K}$
to $\End_{U_q(\mathfrak{sp}_{2m})}\big(V^{\otimes n}/V^{\otimes n}\mathfrak{B}^{(f)}_{n,K} \big)$ is always surjective.
\end{abstract}

\maketitle

\section{Introduction}\label{xxsec1}

Let $m,n\in \mathbb{N}$ and $K$ be an infinite field. Let $V$ be a $2m$-dimensional
symplectic space over $K$ equipped with a non-degenerate skew symmetric bilinear
form $\langle\ ,\ \rangle$.
Let $\mathfrak{B}_n:=\mathfrak{B}_n(-2m)$ be the Brauer algebra over $K$ with the standard
generators $s_1,\ldots,s_{n-1},e_1,\ldots,e_{n-1}$ and parameter $-2m$ (see \cite{DDH} for example). There is a
right action of $\mathfrak{B}_n$ on $V^{\otimes n}$ which commutes with the left
action of the symplectic group $\mathrm{Sp}(V)$. Hence there are two natural $K$-algebra homomorphisms:
$$
\varphi: (\mathfrak{B}_n)^{\rm op}\rightarrow \End_{K\mathrm{Sp}(V)}(V^{\otimes n}),\quad
\psi: K\mathrm{Sp}(V)\rightarrow \End_{\mathfrak{B}_n}(V^{\otimes n}).
$$

The following results are referred to as
Brauer--Schur--Weyl duality for the symplectic group (or Schur--Weyl duality of type $C$).

\begin{theorem}\label{xx1.1}
{\rm (\cite{Brauer},\cite{DePr},\cite{DDH},\cite{Oehms})}
\begin{enumerate}
\item[(1)] Both $\varphi$ and $\psi$ are surjective.

\item[(2)] If $m\geq n$, then $\varphi$ is injective, and hence an isomorphism.
\end{enumerate}
\end{theorem}

For each integer $0\leq f\leq \left\lfloor \frac n2 \right\rfloor$, let $\mathfrak{B}^{(f)}_n$
be the two-sided ideal of the Brauer algebra $\mathfrak{B}_n$ generated by
$e_1e_3\cdots e_{2f-1}$. By convention, we write $\mathfrak{B}^{(0)}_n=\mathfrak{B}_n$ and
$\mathfrak{B}^{(\left\lfloor \frac n2 \right\rfloor+1)}_n=0$. Define $\mathcal{HT}^{\otimes n}_f$ to be the space of
{\em partially harmonic tensors of valence f} as follows
$$
\mathcal{HT}^{\otimes n}_f:=\left\{v\in V^{\otimes n}\mathfrak{B}^{(f)}_n\ \left|\ vx=0,
\ \forall x\in \mathfrak{B}^{(f+1)}_n\right\}\right..
$$
We would like to remark that $\mathcal{HT}^{\otimes n}_0$ is also known as the subspace of {\em traceless tensors} or
{\em harmonic tensors} (see \cite{DeSt}, \cite[\S 10.2.1]{GoodW} and \cite[Corollary 2.6]{Hu}).
Note that these spaces play an important role in the invariant theory of classical groups \cite[Chapter 10]{GoodW}.

Since $\mathfrak{B}_n/\mathfrak{B}^{(1)}_n\cong K\mathfrak{S}_n$, the group algebra of the symmetric
group $\mathfrak{S}_n$, the right action of $\mathfrak{B}_n$ on $V^{\otimes n}$ gives rise to a right action of
$K\mathfrak{S}_n$ on $\mathcal{HT}^{\otimes n}_0$ and hence a $K$-algebra homomorphism
$$
\varphi_1: (K\mathfrak{S}_n)^{\rm op}\rightarrow \End_{K\mathrm{Sp}(V)}(\mathcal{HT}^{\otimes n}_0).
$$
When $K$ is algebraically closed, De Concini and Strickland \cite{DeSt} proved that the dimension of $\mathcal{HT}^{\otimes n}_0$
is independent of the field $K$ and $\varphi_1$ is always surjective. Furthermore, they showed that
$\varphi_1$ is an isomorphism if $m\geq n$ and described $\Ker \varphi_1$ when $m<n$.
On the other hand, using the rational representations of symplectic groups, Maliakas \cite{Mali}
proved that $(\mathcal{HT}^{\otimes n}_0)^{\ast}$ has a good filtration whenever $m\geq n$ and
he claimed that this kind of filtration exists for arbitrary $m$.

Let $K$ be algebraically closed, for any integer $0\leq f\leq \left\lfloor \frac n2 \right\rfloor$, in \cite[Theorem 1.5]{Hu}
Hu proved that
both $V^{\otimes n}/V^{\otimes n}\mathfrak{B}^{(f)}_n$ and $V^{\otimes n}\mathfrak{B}^{(f)}_n$
have a good filtration as $\mathrm{Sp}(V)$-modules. Furthermore, he showed that there is a
$(K\mathrm{Sp}(V),\mathfrak{B}_n/\mathfrak{B}^{(f+1)}_n)$-bimodule isomorphism
$V^{\otimes n}\mathfrak{B}^{(f)}_n/V^{\otimes n}\mathfrak{B}^{(f+1)}_n\cong
(\mathcal{HT}^{\otimes n}_f)^{\ast}$, which gave a positive answer for the claim of Maliakas \cite{Mali}.
Because of the just mentioned bimodule isomorphism, we usually call $V^{\otimes n}\mathfrak{B}^{(f)}_n/V^{\otimes n}\mathfrak{B}^{(f+1)}_n$
the space of {\em dual partially harmonic tensors of valence f}.
Notice that the algebra homomorphism $\varphi$ naturally induces a $K$-algebra homomorphism for $1\leq f\leq \left\lfloor \frac n2 \right\rfloor$
$$
\varphi_f: \big(\mathfrak{B}_n/\mathfrak{B}^{(f)}_n\big)^{\rm op}\rightarrow
\End_{K\mathrm{Sp}(V)}\big(V^{\otimes n}/V^{\otimes n}\mathfrak{B}^{(f)}_n \big).
$$
Motivated by the main results of \cite{DeSt,Mali}, Hu proved the
following Schur--Weyl duality for partially harmonic tensors in
\cite[Theorem 1.8 and Proposition 5.3]{Hu}.

\begin{theorem}\label{xx1.2}
For any infinite field $K$, and any integer $f$ with $1\leq f\leq\left\lfloor \frac n2 \right\rfloor$,
\begin{enumerate}
\item[(1)] the dimension of the endomorphism algebra $\End_{K\mathrm{Sp}(V)}\big(
V^{\otimes n}/V^{\otimes n}\mathfrak{B}^{(f)}_n \big)$ is independent of $K$;

\item[(2)] the homomorphism $\varphi_f$ is surjective.
\end{enumerate}
\end{theorem}

The main motivation of this paper is to show a quantized version of  Theorem~\ref{xx1.2},
which was left as a desired unsolved problem by Hu (see the second paragraph and footnote on \cite[Page 335]{Hu}).
To this end, we first need to recall the integral version of quantized Schur--Weyl duality of type $C$ \cite{Hu11}.

Let $q$ be an indeterminate over $\mathbb{Z}$ and $\mathscr{A}:=\mathbb{Z}[q,q^{-1}]$ be the Laurent
polynomial ring in $q$. Let $U_{\mathbb{Q}(q)}(\mathfrak{sp}_{2m})$ be the quantized enveloping algebra
of the symplectic Lie algebra $\mathfrak{sp}_{2m}$ over the rational function field $\mathbb{Q}(q)$ (see \cite{Lus}).
Let $U_{\mathscr{A}}(\mathfrak{sp}_{2m})$ be the Lusztig's $\mathscr{A}$-form in $U_{\mathbb{Q}(q)}(\mathfrak{sp}_{2m})$.
For any infinite field $K$, $\zeta\in K^{\times}$, the specialization $q\mapsto \zeta$ makes $K$ into
an $\mathscr{A}$-module. For convenience, we always identify $q$ with its natural image $\zeta$ in $K$
and define $U_{q}(\mathfrak{sp}_{2m}):=U_{\mathscr{A}}(\mathfrak{sp}_{2m})\otimes_{\mathscr{A}} K$.

Let $V$ be a $2m$-dimensional symplectic space over $K$ equipped with a non-degenerate skew symmetric bilinear
form $\langle\ ,\ \rangle$. Then $V$,  and hence $V^{\otimes n}$, become a left $U_{q}(\mathfrak{sp}_{2m})$-module
(we refer the reader to \cite[Section 2]{Hu11} for the explicit action).
For any commutative $\mathscr{A}$-algebra $R$, we denote by $\mathfrak{B}_{n,R}$ the Birman--Murakami--Wenzl algebra (BMW algebra for short) over $R$.
 There is a right action of the specialized BMW algebra $\mathfrak{B}_{n,K}$ on the symplectic
tensor space $V^{\otimes n}$ which commutes with the left action of $U_{q}(\mathfrak{sp}_{2m})$.
We refer the reader to Section \ref{xxsec2} for the precise definition of $\mathfrak{B}_{n,K}$
and its action on $V^{\otimes n}$. Let $\varphi$ and $\psi$ be the natural $K$-algebra homomorphisms
$$
\varphi: (\mathfrak{B}_{n,K})^{\rm op}\rightarrow \End_{U_{q}(\mathfrak{sp}_{2m})}(V^{\otimes n}),$$
$$
\psi: U_{q}(\mathfrak{sp}_{2m})\rightarrow \End_{\mathfrak{B}_{n,K}}(V^{\otimes n})
$$
respectively. Moreover we have
\begin{theorem}\label{xx1.3}
{\rm (\cite{CP},\cite{Haya},\cite{Hu11})}
\begin{enumerate}
\item[(1)] Both $\varphi$ and $\psi$ are surjective.

\item[(2)] If $m\geq n$, then $\varphi$ is injective, and hence an isomorphism.
\end{enumerate}
\end{theorem}

For each integer $0\leq f\leq\left\lfloor \frac n2 \right\rfloor$, let $\mathfrak{B}^{(f)}_{n,K}$
be the two-sided ideal of the specialized BMW algebra $\mathfrak{B}_{n,K}$ generated by
$E_1E_3\cdots E_{2f-1}$. It is clear that the algebra homomorphism $\varphi$ in Theorem \ref{xx1.3}
naturally induces a $K$-algebra homomorphism
$$
\varphi_f: \big(\mathfrak{B}_{n,K}/\mathfrak{B}^{(f)}_{n,K} \big)^{\rm op}\rightarrow
\End_{U_q(\mathfrak{sp}_{2m})}\big(V^{\otimes n}/V^{\otimes n}\mathfrak{B}^{(f)}_{n,K} \big).
$$
In this paper, we shall  prove that the algebra homomorphism $\varphi_f$ is actually surjective.

	The paper is organized as follows.
	Section~\ref{xxsec2} lays the foundation by reviewing the diagram category of framed tangles
	and the construction of the type $C$ Reshetikhin--Turaev functor, which underlies the
	diagrammatic realization of the specialized BMW algebra $\mathfrak{B}_{n,K}$ on the tensor space
	and hence the quantized Schur--Weyl duality in our setting.
	Section~\ref{xxsec3} is devoted to the study of the submodules $V^{\otimes n}\mathfrak{B}^{(f)}_{n,K}$.
	The first main result is the identification
	$V^{\otimes n}\mathfrak{B}^{(f)}_{n,K}=\mathcal{O}_{\pi_f}(V^{\otimes n})$, which is obtained by constructing
	a good filtration of $V^{\otimes n}$ compatible with the BMW-ideal filtration.
	This key structural theorem implies that these submodules and their quotients also admit good filtrations;
	moreover, as consequences we derive the basic properties of the corresponding quantum partially harmonic tensors,
	including the existence of Weyl filtrations.
	Building on this structural understanding, Section~\ref{xxsec4} proves the second main result:
	the surjectivity of the natural homomorphism $\varphi_f$, thereby establishing the quantized
	Schur--Weyl duality for the space of dual partially harmonic tensors.

\section{Preliminaries}\label{xxsec2}

In this section, we recall a topological categorification of the quantized Schur--Weyl duality of type $C$
appearing in \cite{XYZ}, which forms one of the crucial tools of this paper when describing the action of a tangle on tensor space.

\subsection{Diagram category of framed tangles}\label{xsec2.1}

The diagram category of framed tangles was introduced in \cite{FY89}, independently known as the category of
non-directed ribbon graphs \cite{RT}.

A {\em tangle} is a knot diagram inside a rectangle (smoothly embedded in $\mathbb{R}^3$) consisting of a finite number
of vertices in the top and the bottom row of the rectangle (the number of vertices in each row may differ)
and a finite number of arcs inside the rectangle such that each vertex is connected to another
vertex by exactly one arc, and arcs either connect two vertices or are closed curves.
Two tangles are {\em regularly isotopic} if they are equivalent by a sequence of the following
Reidemeister Moves II and III (RII and RIII for short) and the isotopies fixing the boundary of the rectangle
(or any rotation of them in a local portion of the rectangle).
	\begin{figure}[ht]
	\begin{equation*}
		\begin{array}{c}
			\begin{tikzpicture}[scale=0.8]
				\begin{knot}[
					clip width=7,
					%flip crossing=0,
					]
					\strand[thick] ((-0.4,2) .. controls (0.4,1.6) and (0.4,0.4) ..  (-0.4,0); 
					\strand[thick] (0.5,2) .. controls (-0.3,1.6) and (-0.3,0.4) ..  (0.5,0); 
				\end{knot}
				\node at (-1.8,1) { RII: };
			\end{tikzpicture}
		\end{array}
		~~=~~
		\begin{array}{c}
			\begin{tikzpicture}[scale=0.8]
				\begin{knot}[
					clip width=5,
					%	flip crossing=0,
					]
					\strand[thick] ((-0.4,2) .. controls (0.4,1.6) and (0.4,0.4) ..  (-0.4,0); 
					\strand[thick] (1.2,2) .. controls (0.4,1.6) and (0.4,0.4) ..  (1.2,0); 
				\end{knot}
			\end{tikzpicture}
		\end{array}
	\end{equation*}
\end{figure}

\begin{figure}[ht]
	\begin{equation*}
		\begin{array}{c}
			\begin{tikzpicture}[scale=0.8]
				\begin{knot}[
					clip width=7,
					%flip crossing=0,
					]
					\strand[thick] (-0.6,2.2) -- (0.6,0); 
					\strand[thick] (0.6,2.2) -- (-0.6,0);
					\strand[thick] (-0.8,1) .. controls (-0.6,2) and (0.6,2) ..  (0.8,1); 
				\end{knot}
				\node at (-1.8,1) { RIII: };
			\end{tikzpicture}
		\end{array}
		~~=~~
		\begin{array}{c}
			\begin{tikzpicture}[scale=0.8,rotate=180]
				\begin{knot}[
					clip width=7,
					%	flip crossing=0,
					]
					\strand[thick] (-0.8,1) .. controls (-0.6,2) and (0.6,2) ..  (0.8,1); 
					\strand[ thick] (-0.6,2.2) -- (0.6,0); 
					\strand[thick] (0.6,2.2) -- (-0.6,0);
					\flipcrossings{1,2,3}
				\end{knot}
			\end{tikzpicture}
		\end{array}
	\end{equation*}
\end{figure}

We denote by $T(s,t)$ the set of all tangles with $s$ vertices
in the top row and $t$ vertices in the bottom row subject to the relations of regular isotopy, named the modified Reidemeister Move I (RI).

\begin{figure}[ht]
	\begin{tikzpicture}[scale=0.6]
		\begin{knot}[clip width=10, flip crossing=2,clip radius=15pt, consider self intersections, end tolerance=3pt]
				\node at (-2.8,2) { RI: };
			\strand[thick] (0,0)
			to[out=up, in=down] (0,0.7) 
			to[out=up, in=right] (-0.5,1.5)
			to[out=left, in=up]  (-1,1)
			to[out=down, in=left]  (-0.5,0.5)
			to[out=right, in=down]  (0,1.3)
			to[out=up, in=down]  (0,2)
			%%%
			to[out=up, in=down] (0,2.7) 
			to[out=up, in=right] (-0.5,3.5)
			to[out=left, in=up]  (-1,3)
			to[out=down, in=left]  (-0.5,2.5)
			to[out=right, in=down]  (0,3.3)
			to[out=up, in=down]  (0,4);	
		\end{knot}
	\end{tikzpicture}~~~~
	\begin{tikzpicture}[scale=0.6]
		\begin{knot}[clip width=10, clip radius=15pt, consider self intersections, end tolerance=3pt]
			\strand[thick] (0,0)
			to[out=up, in=down]  (0,4);
		\end{knot}
		\node at (-0.7,2) {~~~~~~~~~~ \mbox{$=$} };
	\end{tikzpicture}
\end{figure}

There are two operations on tangles:
\begin{eqnarray}\label{eq2.1}
	\circ:& & T(s,t)\times T(t,l)\longrightarrow T(s,l),\\
	\otimes :& & T(s,t)\times T(l,m)\longrightarrow T(s+l,t+m).
\end{eqnarray}
The {\em composition} $\circ$ is defined by concatenation of tangles, reading
from up to down for our late convenience,
and the {\em tensor product} $\otimes$ is given by juxtaposition of tangles. More explicitly,
$D\otimes D'$ means placing $D'$ on the right of $D$ without overlapping.
In order to define the diagram category of framed tangles (see Definition \ref{dc of framed tangles} below,
which is \cite[Definition 2.5]{XYZ})
to be a strict monoidal category, we need to add
an element $\id_0$ in $T(0,0)$ and define the corresponding composition and tensor product as follows:
for any $f\in T(0,t),$ $g\in T(s,0)$ and $h\in  T(s,t),$
\begin{eqnarray}\label{eq2.4}
	&&\id_0\circ f=f,\ \  g\circ \id_0=g,\\
	&&h\otimes \id_0=\id_0\otimes h=h.
\end{eqnarray}

Let $K$ be an infinite field\footnote{Although the base field in \cite{XYZ} is the rational
function field $\mathbb{Q}(r,q)$ with $r,q$ being indeterminates, all the results introduced here
can keep invariant over an arbitrary infinite field $K$.}. Let $r,q\in K\backslash\{0,1\}$. For any
$s,t\in \mathbb{N}_0:=\{0,1,2,\ldots\}$, we denote by $\mathcal{D}(s,t)$ the $K$-linear span of $T(s,t)$
subject to the following relations, (see \cite[Definition 2.4]{XYZ}):

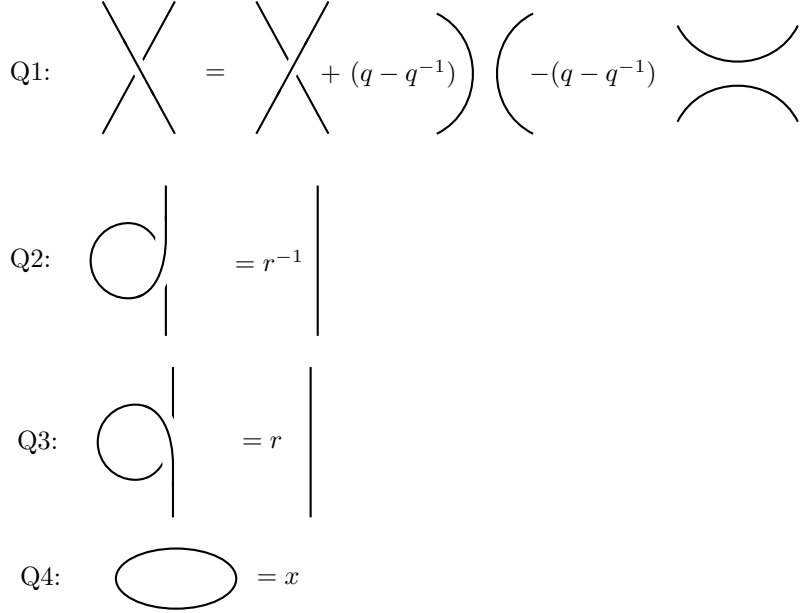
\begin{figure}[ht] 
	% 使用[t]确保放在下一页顶部
 % 强制左对齐	
	% Q1 部分
	\begin{equation*}
% 公式内左对齐
		\begin{array}{c}
			\begin{tikzpicture}[scale=0.8] 
				\begin{knot}[clip width=7]
					\strand[thick] (-0.6,2.2) -- (0.6,0); 
					\strand[thick] (0.6,2.2) -- (-0.6,0);
				\end{knot}
				\node at (-1.8,1) { Q1: };
			\end{tikzpicture}
		\end{array}
		~~=~~
		\begin{array}{c}
			\begin{tikzpicture}[scale=0.8]
				\begin{knot}[clip width=7, flip crossing=1]
					\strand[thick] (-3,2.2) -- (-1.8,0); 
					\strand[thick] (-1.8,2.2) -- (-3,0);
				\end{knot}
				\node at (-0.8,1) { $+~ (q-q^{-1}   )$ };
				\begin{knot}[clip width=5]
					\strand[thick] ((0,2) .. controls (0.8,1.6) and (0.8,0.4) ..  (0,0); 
					\strand[thick] (1.6,2) .. controls (0.8,1.6) and (0.8,0.4) ..  (1.6,0); 
				\end{knot}
				\node at (2.6,1) { $- (q-q^{-1}   )$ };
				\begin{knot}[clip width=5]
					\strand[thick] ((6,0.2) .. controls (5.6,1) and (4.4,1) ..  (4, 0.2); 
					\strand[thick] (6,1.8) .. controls (5.6,1) and (4.4,1) ..  (4,1.8); 
				\end{knot}
			\end{tikzpicture}
		\end{array}
	\end{equation*}
	
	\vspace{10pt} % 添加垂直间距
			\hspace*{-6.5cm}
	% Q2 部分
	% 防止缩进
	\begin{tikzpicture}[baseline=(current bounding box.center)]
		\begin{knot}[clip width=10, flip crossing=1, clip radius=15pt, consider self intersections, end tolerance=3pt]
			\strand[thick] (0,2)
			to[out=up, in=down] (0,2.7) 
			to[out=up, in=right] (-0.5,3.5)
			to[out=left, in=up]  (-1,3)
			to[out=down, in=left]  (-0.5,2.5)
			to[out=right, in=down]  (-0,3.3)
			to[out=up, in=down]  (0,4);    
		\end{knot}
		\node at (-1.8,3) { Q2: };
	\end{tikzpicture}%
	~~~~% 保持原有间距
	\begin{tikzpicture}[baseline=(current bounding box.center)]
		\begin{knot}[clip width=10, clip radius=15pt, consider self intersections, end tolerance=3pt]
			\strand[thick] (0,0)
			to[out=up, in=down]  (0,2);
		\end{knot}
		\node at (-0.7,1) {~~~~~~~~~~ {$= r^{-1} $} };
	\end{tikzpicture}
	
	\vspace{10pt} % 添加垂直间距
			\hspace*{-6.5cm}
	% Q3 部分
% 防止缩进
	\begin{tikzpicture}[baseline=(current bounding box.center)]
		\begin{knot}[clip width=10, flip crossing=2, clip radius=15pt, consider self intersections, end tolerance=3pt]
			\strand[thick] (0,0)
			to[out=up, in=down] (0,0.7) 
			to[out=up, in=right] (-0.5,1.5)
			to[out=left, in=up]  (-1,1)
			to[out=down, in=left]  (-0.5,0.5)
			to[out=right, in=down]  (-0,1.3)
			to[out=up, in=down]  (-0,2);
		\end{knot}
		\node at (-1.8,1) { Q3: };
	\end{tikzpicture}%
	~~~~% 保持原有间距
	\begin{tikzpicture}[baseline=(current bounding box.center)]
		\begin{knot}[clip width=10, clip radius=15pt, consider self intersections, end tolerance=3pt]
			\strand[thick] (0,0)
			to[out=up, in=down]  (0,2);
		\end{knot}
		\node at (-0.7,1) {~~~~~~~~~~ {$= r $} };
	\end{tikzpicture}
	
	\vspace{10pt} % 添加垂直间距
			\hspace*{-6.5cm}
	% Q4 部分
% 防止缩进
	\begin{tikzpicture}[baseline=(current bounding box.center)]
		\draw[thick]  (0,0) ellipse (0.8cm and 0.4cm);
		\node at (1.3,0) {~~~~~~~~~~ {$= x $} };
		\node at (-1.8,0) { Q4: };
	\end{tikzpicture}
\caption{Relations of regular isotopy}
\end{figure}

where $x=1+\frac{r-r^{-1}}{q-q^{-1}}.$ In other words, $\mathcal{D}(s,t)$ is linearly spanned by
the set of all tangles with $s$ vertices
in the top row and $t$ vertices in the bottom row subject to the relations of regular isotopy and the relations
$Q1-Q4$ (see Figure 1.).

\begin{definition}\label{dc of framed tangles}
The {\em diagram category of framed tangles} with parameters $r,q$, denoted $\mathbb{D}(r,q)$,
is a strict monoidal category with objects $\mathbb{N}_0=\{0,1,2,\ldots\}$ and morphisms
$\Hom_{\mathbb{D}(r,q)}(s,t)=\mathcal{D}(s,t)$. The tensor product on objects is defined by $m\otimes n:=m+n$,
and the composition $\circ$ and tensor product $\otimes$
on morphisms are inherited naturally from those of tangles.
\end{definition}

We here use the terminologies related to monoidal categories without interpretation, and
refer the reader to the monograph \cite{EGNO} for the precise meanings if needed.
The following five tangles:
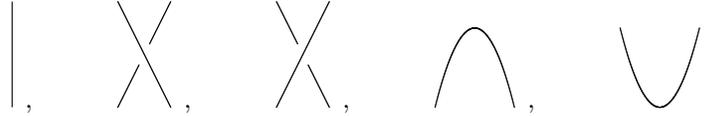
\begin{center}
	\begin{picture}(280, 40)(-5,0)
		\put(0, 0){\line(0, 1){40}}
		\put(5, 0){,}
		
		% \put(40, 0){\line(1, 2){8}}
		\qbezier(40,0)(44,8)(48,16)
		
		% \put(60, 0){\line(-1, 2){20}}
		\qbezier(60,0)(50,20)(40,40)
		
		% \put(60, 40){\line(-1, -2){8}}
		\qbezier(60,40)(56,32)(52,24)
		
		\put(65, 0){,}
		
		% \put(100, 40){\line(1, -2){8}}
		\qbezier(100,40)(104,32)(108,24)
		
		% \put(100, 0){\line(1, 2){20}}
		\qbezier(100,0)(110,20)(120,40)
		
		% \put(120, 0){\line(-1, 2){8}}
		\qbezier(120,0)(116,8)(112,16)
		
		\put(125, 0){,}
		
		\qbezier(160, 0)(175, 60)(190, 0)
		\put(195, 0){,}
		\qbezier(230, 30)(245, -30)(260, 30)
		\put(265, 0){,}
	\end{picture}
\end{center}

are called the {\em elementary} tangles and denoted by $I, X, X^{op}, A, U$
respectively. The category
$\mathbb{D}(r,q)$ has a contravariant functor $^*:\mathbb{D}(r,q)\to \mathbb{D}(r,q)$
such that $n^*=n$ for any object $n\in\mathbb{N}$. In order to get a clear picture of morphisms
under the functor $^*$, we define the following useful tangles. Let $U_n: n\otimes n\to 0$ and $A_n: 0\to n\otimes n$
be tangles defined by
\begin{eqnarray}\label{eq2.5}
U_n&=&(I^{\otimes (n-1)}\otimes U\otimes I^{\otimes(n-1)})\circ\cdots\circ (I\otimes U\otimes I)\circ U,\\
A_n&=&A\circ(I\otimes A\otimes I)\circ\cdots\circ (I^{\otimes (n-1)}\otimes A\otimes I^{\otimes (n-1)}).
\end{eqnarray}
These are depicted as follows
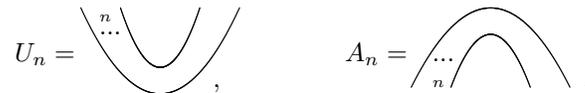
\begin{center}
\begin{picture}(250, 40)(-5,0)
\put(0, 10){$U_n=$}
\qbezier(25, 30)(55, -35)(85, 30)
\put(32, 25){\tiny$n$}
\put(32, 20){...}
\qbezier(40, 30)(55, -15)(70, 30)
\put(75, 0){,}

\put(125, 10){$A_n=$}
\qbezier(150, 0)(180, 60)(210, 0)
\put(158, 10){...}
\put(158, 0){\tiny$n$}
\qbezier(165, 0)(180, 40)(195, 0)
\put(215, 0){.}

\end{picture}
\end{center}
Let $I_n:=I^{\otimes n}$ be the identity morphism $\id_n$. The linear mapping
$^*: \mathcal{D}(s,t)\to \mathcal{D}(t,s)$
defined for any tangle $D\in \mathcal{D}(s,t)$ by $D^*:=(I_t\otimes A_s)\circ
(I_t\otimes D\otimes I_s)\circ (U_t\otimes I_s)$ can be described as follows
%\begin{figure}[h]
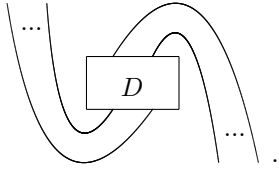
\begin{center}
\begin{picture}(100, 60)(-20,0)
\put(-5, 20){\line(1, 0){35}}
\put(-5, 20){\line(0, 1){20}}
\put(30, 20){\line(0, 1){20}}
\put(-5, 40){\line(1, 0){35}}
\put(8, 25){$D$}

\qbezier(5, 40)(40, 95)(60, 0)
\qbezier(20, 40)(35, 70)(45, 0)
\put(47, 10){...}

\qbezier(5, 20)(-15, -10)(-20, 60)
\qbezier(20, 20)(-20, -35)(-35, 60)
\put(-30, 50){...}
\put(65, 0){.}
\end{picture}
\end{center}
%\caption{$D^*$}
%\end{figure}
Then the contravariant functor $^*$ gives rise to the duality of
$\mathbbm{D}(r,q)$ with evaluation $U$ and coevaluation $A$. It follows from
\cite[Proposition 2.6 and Theorem 2.7]{XYZ} that the strict monoidal category $\mathbb{D}(r,q)$ is pivotal
and any morphism of $\mathbb{D}(r,q)$ is generated by four elementary tangles
$I, X, A, U$ through linear combination, composition and tensor product.

For any object $n>0$, the set of morphisms $\mathcal{D}(n,n)$ forms a unital
associative $K$-algebra under the composition of tangles. This is the
Kauffman's tangle algebra \cite{Ka} which is isomorphic to a BMW algebra.

\begin{definition}{\rm (\cite{BW,Mu})} \label{def BMW}
The generic \emph{BMW algebra} $\mathfrak{B}_n(r,q)$ is
a unital associative $K$-algebra generated by the
elements $T_i^{\pm1}$ and $E_i$ for $1\leq i\leq n-1$ subject to the
relations:
  \begin{align*}
    T_i-T_i^{-1}&=(q-q^{-1})(1-E_i), &\quad&\text{ for } 1\leq i\leq n-1,\\
    E_i^2&=x E_i, &&\text{ for } 1\leq i\leq n-1,\\
    T_iT_{i+1} T_i&=T_{i+1} T_iT_{i+1}, &&\text{ for } 1\leq i\leq n-2,\\
    T_iT_j&=T_jT_i, &&\text{ for }|i-j|>1,\\
    E_iE_{i+1} E_i&=E_i,\;  E_{i+1}E_i E_{i+1}=E_{i+1},
    &&\text{ for } 1\leq i\leq n-2,\\
    T_iT_{i+1}E_i&=E_{i+1} E_i,\; T_{i+1}T_iE_{i+1}=E_iE_{i+1}, &&\text{
      for } 1\leq i\leq n-2,\\
    E_iT_i&=T_iE_i=r^{-1}E_i &&\text{ for } 1\leq i\leq n-1,\\
    E_iT_{i+1}E_i&=rE_i,\;E_{i+1}T_iE_{i+1}=rE_{i+1}, &&\text{ for }
    1\leq i\leq n-2,
      \end{align*}
where $x=1+\dfrac{r-r^{-1}}{q-q^{-1}}$.
\end{definition}

By the following well-known isomorphism (see \cite{MW}) between the Kauffman's tangle algebra $\mathcal{D}(n,n)$
and the generic BMW algebra $\mathfrak{B}_n(r,q)$, we can change freely between algebraic
formulations and tangle diagrams.
%{\rm (see \cite[Theorem 2.9]{XYZ})}
\begin{theorem}  \label{iso between tangle and BMW}
The Kauffman's tangle algebra $\mathcal{D}(n,n)$ is isomorphic to the BMW algebra $\mathfrak{B}_n(r,q)$
with an isomorphism given by
$$
T_i\mapsto I_{i-1}\otimes X\otimes I_{n-1-i},\ \ E_i\mapsto I_{i-1}\otimes (U\circ A)\otimes I_{n-1-i}.
$$
\end{theorem}

\vspace{6pt}
\subsection{Type $C$ Reshetikhin--Turaev functor}\label{xsec2.2}

Let $K$ be an infinite field and $q\in K\backslash\{0,1\}$. Let $V$ be the natural
representation of the quantized enveloping algebra $U_{q}(\mathfrak{sp}_{2m})$ (see the Introduction of this paper).
For each integer $i$ with
$1\leq i\leq 2m$, set $i':=2m+1-i$. We fix an ordered basis
$\{v_i\}_{i=1}^{2m}$ of $V$ such that
$$
\langle v_i,v_j\rangle=0=\langle v_{i'},v_{j'}\rangle,\quad
\langle v_i,v_{j'}\rangle=\delta_{ij}=-\langle v_{j'},v_i\rangle,\ \ \forall 1\leq i,j\leq m.
$$
We set
$$
(\rho_1,\cdots,\rho_{2m}):=(m,m-1,\cdots,1,-1,\cdots,-m+1,-m),
$$
and $\epsilon_i:=\sign(\rho_i)$. For any $i,j\in\{1,2,\cdots,2m\}$, we set $E_{i,j}\in\End_{K}(V)$ to be
the elementary matrix whose entries are all zero except $1$ for the $(i,j)$-th entry.
Let us define (see \cite[Section 3]{Hu11} for the same notations)
$$\begin{aligned} \beta' &:=\sum_{1\leq i\leq
2m}\Bigl(qE_{i,i}\otimes E_{i,i}+q^{-1}E_{i,i'}\otimes
E_{i',i}\Bigr)+\sum_{\substack{1\leq
i,j\leq 2m\\ i\neq j,j'}} E_{i,j}\otimes E_{j,i}+\\
&\qquad\qquad (q-q^{-1})\sum_{1\leq i<j\leq 2m}\Bigl(E_{i,i}\otimes
E_{j,j}-q^{\rho_j-\rho_i}\epsilon_i\epsilon_j E_{i,j'}\otimes
E_{i',j}\Bigr),\\
\gamma' &:=\sum_{1\leq i,j\leq
2m}q^{\rho_j-\rho_i}\epsilon_i\epsilon_j E_{i,j'}\otimes E_{i',j}.
\end{aligned}
$$
For $i=1,2,\ldots,n-1$, we set
$$
\beta'_i:=\id_{V^{\otimes i-1}}\otimes \beta'\otimes \id_{V^{\otimes n-i-1}},\quad
\gamma'_i:=\id_{V^{\otimes i-1}}\otimes \gamma'\otimes \id_{V^{\otimes n-i-1}}.
$$
Then the mapping which sends the generator
$T_i$ to $\beta'_i$ and $E_i$ to $\gamma'_i$ can be naturally extended to a right action
of the specialized BMW algebra $\mathfrak{B}_{n,K}$ on tensor space $V^{\otimes n}$,
which commutes with the left action of $U_q(\mathfrak{sp}_{2m})$ on $V^{\otimes n}$.

In order to formulate the type $C$ Reshetikhin--Turaev functor (RT-functor for short) explicitly \cite{RT},
we define the following $K$-linear maps:
$$\begin{aligned}
\check{R}: V\otimes V\longrightarrow V\otimes V,&\quad\quad v\otimes w\mapsto \beta'(v\otimes w),\\
C: K\longrightarrow V\otimes V,&\quad\quad 1\mapsto \alpha:=\sum_{1\leq k\leq
2m}q^{-\rho_k}\epsilon_k v_k\otimes v_{k'},\\
E: V\otimes V\longrightarrow K,&\quad\quad v_i\otimes v_j\mapsto q^{-\rho_i}\epsilon_j(v_i,v_j).&
\end{aligned}
$$
Clearly $\gamma'=E\circ C$, where the composition reads from left to right,
because of the {\em right} action of $\mathfrak{B}_{n,K}$ on $V^{\otimes n}$.

\begin{definition}\label{category of tensor rep}
We denote by $\mathcal{T}(V)$ the full subcategory of $U_q(\mathfrak{sp}_{2m})$-modules
with objects $V^{\otimes n}$, where $n\in\mathbb{N}_0$ and $V^{\otimes 0}=K$
by convention. The usual tensor product of $U_q(\mathfrak{sp}_{2m})$-modules and of
$U_q(\mathfrak{sp}_{2m})$-homomorphisms is a bi-functor $\mathcal{T}(V)\times \mathcal{T}(V)
\to \mathcal{T}(V)$, which will be called the tensor product of this category. We call
$\mathcal{T}(V)$ the {\em category of tensor representations} of $U_q(\mathfrak{sp}_{2m})$.
\end{definition}

Since $V\cong V^*$ as $U_q(\mathfrak{sp}_{2m})$-modules, the category $\mathcal{T}(V)$
is also a pivotal strict monoidal category with the evaluation $E$ and the coevaluation $C$.
Under the duality decided by $E$ and $C$, $V$ and $V^*$ are the same object in the category $\mathcal{T}(V)$.

Let $\mathbb{D}(\mathfrak{sp}_{2m})$ be the specialized diagram category of framed tangles $\mathbb{D}(r,q)$ with
$r=-q^{2m+1}$. Then we have

\begin{theorem}\label{RT-functor}
{\rm (\cite[Theorem 3.3]{XYZ})}
There exists a unique additive covariant functor $F: \mathbb{D}(\mathfrak{sp}_{2m})\rightarrow
\mathcal{T}(V)$ satisfying the following properties:

\begin{enumerate}
\item[(1)] $F$ sends the object $n$ to $V^{\otimes n}$ and the morphism $D: s\to t$ to $F(D):
V^{\otimes s}\to V^{\otimes t}$, where $F(D)$ is defined on the generators of tangles by
%\begin{eqnarray}\label{FR-functor}
$$\begin{aligned}
F\left(
\begin{picture}(30, 20)(0,0)
\put(15, -15){\line(0, 1){35}}
\end{picture}\right)=\id_V,
\quad\quad&
F\left(
\begin{picture}(30, 20)(0,0)
	% \put(5, -15){\line(1, 2){8}}
	\qbezier(5,-15)(9,-7)(13,1)
	
	% \put(25, -15){\line(-1, 2){18}}
	\qbezier(25,-15)(16,3)(7,21)
	
	% \put(24, 21){\line(-1, -2){7}}
	\qbezier(24,21)(20.5,14)(17,7)
\end{picture}\right) = \check{R}, \\
F\left(
\begin{picture}(30, 20)(0,0)
\qbezier(5, -15)(15, 50)(25, -15)
\end{picture}\right) = C,
\quad\quad&
F\left(
\begin{picture}(30, 20)(0,0)
\qbezier(5, 20)(15, -50)(25, 20)
\end{picture}\right) = E.
\end{aligned}$$
%\end{eqnarray}
\item[(2)] $F$ is a pivotal monoidal functor, i.e. $F$ preserves the tensor products and the dualities.
\end{enumerate}
\end{theorem}

The functor $F$ in Theorem \ref{RT-functor} is called the type $C$ RT-functor.

\begin{lemma}\label{tech lemma}
{\rm (\cite[Lemma 3.4]{XYZ})}
Let $H(s,t):=\Hom_{U_q(\mathfrak{sp}_{2m})}(V^{\otimes s},V^{\otimes t})$ for all $s,t\in\mathbb{N}_0$.
Define the linear maps
$$
\mathrm{U}_s^t:=(-\otimes I_t)\circ (I_s\otimes U_t): \mathcal{D}(n,s+t)\longrightarrow \mathcal{D}(n+t,s),
$$
$$
\mathrm{A}_t^n:=(I_{n}\otimes A_t)\circ (-\otimes I_t): \mathcal{D}(n+t,s)\longrightarrow \mathcal{D}(n,s+t).
$$
Then we have:

\begin{enumerate}
\item[(1)] the $K$-linear maps
$$
F\mathrm{U}_s^t:=(-\otimes \id^{\otimes t}_V)(\id^{\otimes s}_V\otimes F(U_t)): H(n,s+t)\longrightarrow H(n+t,s),
$$
$$
F\mathrm{A}_t^n:=(\id^{\otimes n}_V\otimes F(A_t))(-\otimes \id^{\otimes t}_V): H(n+t,s)\longrightarrow H(n,s+t)
$$
are well-defined and are mutually inverses of each other;

\item[(2)] the functor $F$ induces a linear map
$$
F^s_t: \mathcal{D}(s,t)\longrightarrow H(s,t)=\Hom_{U_q(\mathfrak{sp}_{2m})}(V^{\otimes s},V^{\otimes t}),\quad
D\mapsto F(D),
$$
and the following diagrams are commutative:
\[\xymatrix@C=1.0cm{ \mathcal{D}(n+t,s)\ar[r]^{\mathrm{A}_t^n}
\ar[d]_{F_s^{n+t}}& \mathcal{D}(n,s+t)\ar[d]^{F_{s+t}^n}& \\
  H(n+t,s)\ar[r]^{F\mathrm{A}_t^n} &H(n,s+t),}
    \xymatrix@C=1.0cm{ \mathcal{D}(n,s+t)\ar[r]^{\mathrm{U}_s^t}
\ar[d]_{F_{s+t}^n}& \mathcal{D}(n+t,s)\ar[d]^{F_s^{n+t}}& \\
  H(n,s+t)\ar[r]^{F\mathrm{U}_s^t} &H(n+t,s).}\]
\end{enumerate}
\end{lemma}

The following topological categorification of the quantized Schur--Weyl duality of type $C$
was essentially obtained in \cite[Theorem 4.8]{XYZ}.

\begin{theorem}\label{fullness of RT functor}
The RT-functor $F: \mathbb{D}(\mathfrak{sp}_{2m})\rightarrow
\mathcal{T}(V)$ is full, i.e. $F$ is surjective on $\Hom$ spaces.
Furthermore, the map $F^s_t: \mathcal{D}(s,t)\rightarrow H(s,t)$ is injective
if $s+t\leq 2m$
\end{theorem}

\begin{proof}
It follows from Lemma \ref{tech lemma} and Theorem \ref{xx1.3}.
\end{proof}

\section{Good Filtrations and Truncation Functors}\label{xxsec3}

In this section, we relate the BMW-ideal filtration of $V^{\otimes n}$ to highest weight theory
and truncation functors. From this section on, we fix the notation $\mathbb U:=U_q(\mathfrak{sp}_{2m})$.
Let $K$ be an infinite field.
For $0\le f\le \lfloor n/2\rfloor$, let $\mathfrak{B}^{(f)}_{n,K}$ be the two-sided ideal of
$\mathfrak{B}_{n,K}$ generated by $E_1E_3\cdots E_{2f-1}$.
We show that the layers coming from  the chain of ideals $\{\mathfrak{B}^{(f)}_{n,K}\}$ admit good filtrations and
identify $V^{\otimes n}\mathfrak{B}^{(f)}_{n,K}$ with the truncation submodule
$\mathcal{O}_{\pi_f}(V^{\otimes n})$, from which the basic properties of quantum partially harmonic tensors follow.

\subsection{Good filtrations for dual harmonic tensor}\label{subsec:good-filtration}

With the above notation, we establish good filtrations for the harmonic tensor modules.

Let $k\in\mathbb{N}$. Recall that a composition of $k$ is a sequence of nonnegative integers
$\lambda=(\lambda_1,\lambda_2,\ldots)$ such that $\sum_{i\geq 1}\lambda_i=k$. A composition
$\lambda=(\lambda_1,\lambda_2,\ldots)$ of $k$ is said to be a {\em partition} if
$\lambda_1\geq\lambda_2\geq\cdots$, and in this case, we write $\lambda\vdash k$ and denote by
$\ell(\lambda)$ the largest integer $j$ with $\lambda_j\neq 0$.
Let $E$ be an Euclidian space with standard basis $\varepsilon_1,\varepsilon_2,\ldots,\varepsilon_{2m}$.
Then $S:=\{\alpha_i=\varepsilon_i-\varepsilon_{i+1},\alpha_m=2\varepsilon_m\mid 1\leq i<m\}$ is a set
of simple roots in the root system of type $C_m$. We identify each weight
$\lambda=\lambda_1 \varepsilon_1+\cdots+\lambda_m\varepsilon_m$
with $\lambda=(\lambda_1,\ldots,\lambda_m)\in\mathbb{Z}^m$. Thus a weight $\sum_{i=1}^m \lambda_i \varepsilon_i$
is dominant if and only if $\lambda=(\lambda_1,\lambda_2,\ldots,\lambda_m)$ is a partition of some $k\in\mathbb{N}$.
For any two weights $\lambda,\mu$, we define $\lambda\geq \mu$ if and only if
$\lambda-\mu=\sum_{i=1}^m \mathbb{Z}_{\geq 0}\alpha_i$.

For each dominant weight
$\lambda$, we use $\Delta(\lambda),\nabla(\lambda),$ $L(\lambda)$ to denote the Weyl module,
dual Weyl module (also called Schur module) and irreducible module of $\mathbb{U}$ labelled by
$\lambda$ respectively \cite{APW}.
Let $f$ be an integer with $0\leq f\leq \left\lfloor \frac n2 \right\rfloor$. Set
$$
\pi_f:=\{\lambda\vdash n-2f-2r \mid  \ell(\lambda)\leq m, 0\leq r\leq \left\lfloor \frac n2-f \right\rfloor\}.
$$
It is well-known that $\pi_f$ is the set of all dominant weights of $V^{\otimes n-2f}$
as a $\mathbb{U}$-module (see \cite{HuXiao11} for example). Let $\mathcal{C}(\pi_f)$ denote the category of
finite dimensional $\mathbb{U}$-modules $M$ such that all composition factors
of $M$ have the form $L(\lambda)$ with $\lambda\in\pi_f$.  

%\subsection{Truncation functors​}
We define a functor
$\mathcal{O}_{\pi_f}$ from the category of finite dimensional
$\mathbb{U}$-modules to the category $\mathcal{C}(\pi_f)$ as follows:
$$
\mathcal{O}_{\pi_f}(M)=\sum_{M'\subseteq M, M'\in
	\mathcal{C}(\pi_f)}M'.
$$

As usual, we denote $\mathbb{U}$-mod to be the category of all finite dimensional left $\mathbb{U}$-modules of type 1.
We use $\mathcal{F}(\Delta)$ (resp. $\mathcal{F}(\nabla)$) to denote the full subcategory of $\mathbb{U}$-mod
given by all $\mathbb{U}$-modules having a filtration with all subquotients of this filtration being isomorphic
to some $\Delta(\lambda)$ (resp. $\nabla(\lambda)$) with dominant weight $\lambda$. 
If $M\in \mathcal{F}(\Delta)$, we use $[M:\Delta(\lambda)]$ to present the number of factors
in a filtration which are isomorphic to $\Delta(\lambda)$. Note that the multiplicities $[M:\Delta(\lambda)]$
is independent of the choice of filtration. Similarly, we can define $[M:\nabla(\lambda)]$.
Recall that by a tilting module we mean an object in $\mathcal{F}(\Delta)\cap \mathcal{F}(\nabla)$.

\begin{definition}\label{good filtration}
If $M\in \mathcal{F}(\nabla)$, we say that $M$ has a {\em good filtration}.
\end{definition}

The proof of the following Lemma \ref{weight-functor O} is similar to the analogue
in \cite[Part II, Chapter A]{Jan}

\begin{lemma}\label{weight-functor O}
Let $M\in \mathbb{U}$-{\rm mod}.
\begin{enumerate}
\item[(1)] $M$ belongs to $\mathcal{C}(\pi_f)$ if and only if each dominant weight of $M$
	belongs to $\pi_f$.

\item[(2)] If $M$ has a good filtration, then 
   $\mathcal{O}_{\pi_f}(M)$ has a good filtration with $$
	[\mathcal{O}_{\pi_f}(M):\nabla(\lambda)]=\begin{cases}
		[M:\nabla(\lambda)], &\text{if $\lambda\in\pi_{f}$;}\\
		0, &\text{otherwise.}
	\end{cases}
	$$

\item[(3)] $\mathcal{O}_{\pi_f}$ is a left exact functor.
\end{enumerate}
\end{lemma}

Recall that the notation $\alpha=\sum_{1\leq k\leq 2m}q^{-\rho_k}\epsilon_k v_k\otimes v_{k'}$ in
subsection \ref{xsec2.2}. Then $K\alpha$ is a one dimensional trivial $\mathbb{U}$-submodule of
$V^{\otimes 2}$ (see \cite[Lemma 3.1]{HuXiao11}). For each integer $0\leq f\leq\left\lfloor \frac n2 \right\rfloor$,
let $\mathfrak{B}^{(f)}_{n,K}$
be the two-sided ideal of the specialized BMW algebra $\mathfrak{B}_{n,K}$ generated by
$E_1E_3\cdots E_{2f-1}$. By convention, we write $\mathfrak{B}^{(0)}_{n,K}=\mathfrak{B}_{n,K}$
and $\mathfrak{B}^{(\left\lfloor \frac n2 \right\rfloor+1)}_{n,K}=0$. 

The following lemma shows that the submodule of partial harmonic tensors is preserved by the truncation functor.

\begin{lemma}\label{O preserve partial tensor}
With the notation as above, we have 
$\mathcal{O}_{\pi_f}\bigl(V^{\otimes n}\mathfrak{B}^{(f)}_{n,K}\bigr)=V^{\otimes n}\mathfrak{B}^{(f)}_{n,K}$.
\end{lemma}

\begin{proof}
		$V^{\otimes n}\mathfrak{B}_{n,K}^{(f)}$ is a $\mathbb{U}$-submodule because the actions commute.  The ideal $\mathfrak{B}_{n,K}^{(f)}$ is generated by $E_1E_3\cdots E_{2f-1}$, which corresponds to contracting $f$ tensor pairs to the one-dimensional trivial submodule $K\alpha \subset V^{\otimes 2}$ by \cite[Lemma 3.1]{HuXiao11}. 
	
	Thus, $V^{\otimes n}\mathfrak{B}_{n,K}^{(f)}$ is spanned by elements of the form
	\[
	(\alpha^{\otimes f} \otimes v) T_{\sigma}, \quad v \in V^{\otimes n-2f}, \ \sigma \in \mathfrak{S}_n.
	\]

	Since the right action of $T_{\sigma}$ is $\mathbb{U}$-equivariant, the $\mathbb{U}$-module $V^{\otimes n}\mathfrak{B}_{n,K}^{(f)}$ is isomorphic to
	\[
	V^{\otimes n}\mathfrak{B}_{n,K}^{(f)} \cong \alpha^{\otimes f} \otimes V^{\otimes n-2f} \cong (K\alpha)^{\otimes f} \otimes V^{\otimes n-2f}.
	\]
	As $(K\alpha)^{\otimes f}$ is a trivial $\mathbb{U}$-module, we have
	\[
	(K\alpha)^{\otimes f} \otimes V^{\otimes n-2f} \cong V^{\otimes n-2f}.
	\]
	By definition of $\pi_f$, the dominant weights of $V^{\otimes n-2f}$ lie in $\pi_f$. Therefore, all dominant weights of $V^{\otimes n}\mathfrak{B}_{n,K}^{(f)}$ belong to $\pi_f$. By Lemma \ref{weight-functor O}(1), we conclude
	\[
	V^{\otimes n}\mathfrak{B}_{n,K}^{(f)} \in \mathcal{C}(\pi_f) \quad \text{and} \quad \mathcal{O}_{\pi_f}\bigl(V^{\otimes n}\mathfrak{B}^{(f)}_{n,K}\bigr) = V^{\otimes n}\mathfrak{B}^{(f)}_{n,K}. \qedhere
	\]
\end{proof}

	We want to show that $\mathcal{O}_{\pi_f}\bigl(V^{\otimes
	n}\bigr)=V^{\otimes n}\mbb_n^{(f)}$. Before proving this equality, we
need some preparation. For simplicity, for any two finite
dimensional $\mathbb{U}$-modules $M, N$, we use $\sum_{\phi: M\rightarrow
	N}\Image\phi$ to denote the sum of all the image subspaces
$\Image\phi$, where $\phi$ runs over all the $\mathbb{U}$-module
homomorphisms from $M$ to $N$. It is clearly a $\mathbb{U}$-submodule of $N$.
Let $H(s,t):=\Hom_{\mathbb{U}}(V^{\otimes s},V^{\otimes t})$ for all $s,t\in\mathbb{N}$, and $\mathcal{D}(s,t):=\Hom_{\mathbb{D}(r,q)}(s,t)$.
\begin{lemma} \label{keylem1} 	
	The natural embedding $$ \iota_1: \Hom_{\mathbb{U}}\bigl(V^{\otimes n-2f},
	V^{\otimes n}\mathfrak{B}_{n,K}^{(f)}\bigr)\rightarrow\Hom_{\mathbb{U}}\bigl(V^{\otimes
		n-2f}, V^{\otimes n}\bigr)
	$$
	is actually an isomorphism, and
	$$ \sum_{\phi: V^{\otimes
			n-2f}\rightarrow V^{\otimes n}}\Image\phi=V^{\otimes
		n}\mathfrak{B}_{n,K}^{(f)}=\sum_{\phi: V^{\otimes n-2f}\rightarrow V^{\otimes
			n}\mathfrak{B}_{n,K}^{(f)}}\Image\phi.$$

\end{lemma}

\begin{proof} 
	
	First observe that $V^{\otimes n}\mathfrak{B}_{n,K}^{(f)}$ is a $\mathbb{U}$-submodule of $V^{\otimes n}$ which decomposes as a sum of copies of $V^{\otimes n-2f}$. We hence have 
	$$ V^{\otimes n}\mathfrak{B}_{n,K}^{(f)} = \sum_{\phi: V^{\otimes n-2f}\rightarrow V^{\otimes n}\mathfrak{B}_{n,K}^{(f)}} \Image\phi \subseteq \sum_{\phi: V^{\otimes n-2f}\rightarrow V^{\otimes n}} \Image\phi \subseteq \sum_{F(D), D\in \mathcal{D}(n-2f,n) } \Image F(D). $$
	
	Using the Reshetikhin--Turaev functor $F$ and by Theorem \ref{RT-functor} and \ref{fullness of RT functor}, $H(s,t)$ is linearly spanned by $F(\mathcal{D}(s,t))$. Recall $U$ is an elementary tangle in Definition 2.1, and $F(U) = E$ where $E(v_i \otimes v_j) = q^{-\rho_i}\epsilon_j(v_i,v_j)$ by Theorem \ref{RT-functor}.
	
	Thus, we have the following commutative diagram
	\[
	\begin{tikzcd}
		\mathcal{D}(n-2f,n) \arrow[r, twoheadrightarrow, "F_n^{n-2f}"] \arrow[d, hookrightarrow,"i_1"'] & H(n-2f,n) \arrow[d, hookrightarrow,"i_2"] \\
		\mathcal{D}(n,n) \arrow[r, twoheadrightarrow,"F_n^n"] & H(n,n)
	\end{tikzcd}
	\]
	
	Here, $i_1(D) = U^{\otimes f}\otimes D$ and $i_2(\phi) = E^{\otimes f}\otimes \phi$ defined by
	\[
	i_2(\phi)(v_{i_1}\otimes\cdots \otimes v_{i_n})
	=
	E^{\otimes f}(v_{i_1}\otimes\cdots \otimes v_{i_{2f}})\cdot
	\phi(v_{i_{2f+1}}\otimes\cdots \otimes v_{i_n}).
	\]
	
	Thus, for any $D \in \mathcal{D}(n-2f,n)$, we have:
	\[
	\begin{aligned}
		i_2 F_n^{n-2f}(D) 
		&= i_2\big( F_n^{n-2f}(D) \big) \\
		&= E^{\otimes f}\otimes F_n^{n-2f}(D) \\
		&= F_0^{2}(U)^{\otimes f}\otimes F_n^{n-2f}(D) \\
		&= F_n^{n}\big( U^{\otimes f}\otimes D \big) \\
		&= F_n^n i_1 (D).
	\end{aligned}
	\]
	
	Thus $\Image\big(i_2F_n^{n-2f}(D)\big)=\Image F_n^n(i_1(D)) \subseteq V^{\otimes n}\mathfrak{B}_{n,K}^{(f)}$.
	Choose $u_0\in V^{\otimes 2f}$ such that $E^{\otimes f}(u_0)=1$ (for instance $u_0=\alpha^{\otimes f}$).
	Then for any $x\in V^{\otimes(n-2f)}$ we have
	$F_n^{n-2f}(D)(x)=\big(i_2F_n^{n-2f}(D)\big)(u_0\otimes x)\in V^{\otimes n}\mathfrak{B}_{n,K}^{(f)}$,
	hence $\Image F_n^{n-2f}(D)\subseteq V^{\otimes n}\mathfrak{B}_{n,K}^{(f)}$.
\end{proof}

For later use, we collect some well-known homological properties of $\mathbb{U}$-modules as follows (see \cite{AST,Donkin}).
\begin{lemma} \label{homological properties}
	Let $\lambda,\mu$ be two dominant weights, $i\in\mathbb{Z}_{\geq 0}$.
	\begin{enumerate}
		\item[(1)] $\Ext_{\mathbb{U}}^{i}\bigl(\Delta(\lam),\nabla(\mu)\bigr)=\delta_{i,0}\delta_{\lam,\mu}K$;
		
		\item[(2)] $\Ext_{\mathbb{U}}^{1}\bigl(\nabla(\lam),
		\nabla(\mu)\bigr)\neq 0$ implies that $\lam>\mu$. In particular,
		$$\Ext_{\mathbb{U}}^{1}\bigl(\nabla(\lam), \nabla(\lam)\bigr)=0;$$
		
		\item[(3)] $\Ext_{\mathbb{U}}^{1}\bigl(\Delta(\lam),
		\Delta(\mu)\bigr)\neq 0$ implies that $\lam<\mu$. In particular,
		$$\Ext_{\mathbb{U}}^{1}\bigl(\Delta(\lam), \Delta(\lam)\bigr)=0;$$
		
		\item[(4)] $\Hom_{\mathbb{U}}\bigl(\nabla(\lam),\nabla(\mu)\bigr)\neq 0$ unless $\mu\leq\lam$, and
		$\Hom_{\mathbb{U}}\bigl(\nabla(\lam),\nabla(\lam)\bigr)=K$.
	\end{enumerate}
\end{lemma}

\begin{lemma} \label{BP2} Let $M$ be a $\mathbb{U}$-module,  and $\lam$ be a dominant weight. Suppose $\rho: \nabla(\lam)\rightarrow
	M$ is a surjective $\mathbb{U}$-module homomorphism such that
	$\rho(u_{\lam})\neq 0$ for some highest weight vector
	$u_{\lam}\in\nabla(\lam)$ of weight $\lam$, then $\rho$ is an isomorphism. In
	particular, $M\cong\nabla(\lam)$.
\end{lemma}

\begin{proof} This follows from the fact that
	$\rm{soc}_{\mathbb{U}}\bigl(\nabla(\lam)\bigr)=L(\lam)$.
\end{proof}

\begin{lemma}\label{lem:weight-comparison} {\rm (\cite[Lemma 4.16]{HuXiao21})}
	Let $0 \leq a < b \leq \left\lfloor \frac n2 \right\rfloor$ and $\lambda \vdash n-2a$, $\mu \vdash n-2b$ with $\ell(\lambda), \ell(\mu) \leq m$. Then $\lambda \nleq \mu$.
\end{lemma}

As a $\mathbb{U}$-module, $V^{\otimes n}$ is tilting. We construct a filtration respecting valence and dominance order:

\begin{equation}\label{eq:good-filtration}
	0 = M_0 \subset M_1 \subset \cdots \subset M_p = V^{\otimes n},
\end{equation}
such that:
\begin{enumerate}
	\item[(1)]  $M_i/M_{i-1} \cong \nabla(\lambda^{(i)})^{\oplus n_i}$ where $\lambda^{(i)} \vdash n-2g_i$, $0 \leq g_i \leq \left\lfloor \frac n2 \right\rfloor$, $\ell(\lambda^{(i)}) \leq m;$
	\item[(2)]  $\lambda^{(i)} \neq \lambda^{(j)}$ for $i \neq j$, and $\lambda^{(i)} < \lambda^{(j)} \implies i < j;$ and
	\item[(3)]  for some $1 \leq t \leq p$, $\lambda^{(i)} \in \pi_f \iff 1 \leq i \leq t$ where $\pi_f = \{\lambda \vdash n-2f-2r \mid \ell(\lambda) \leq m, 0 \leq r \leq\left\lfloor \frac n2-f \right\rfloor\}.$
\end{enumerate}

Now we assume that $\lambda$ is a partition of $n-2f$ satisfying
$\ell(\lam)\leq m$. We define maximal vectors \cite[Definition 3.1]{HuXiao11},\label{def:special-vectors} $$ z_{f,\lambda}:=\alpha^{\otimes
	f}\otimes v_{\lambda}T_{w_{\lambda}}Y_{\lambda'},
$$
where $\lam'$ denotes the conjugate of $\lam$,  $Y_{\lam'}=\sum_{w\in\bBS_{\lam'}}(-q)^{-\ell(w)}{T}_w$ and
$$v_{\lambda}:=\underbrace{v_1\otimes\cdots\otimes v_1}_{\text{$\lam_1$ copies}}\otimes
\underbrace{v_2\otimes\cdots\otimes v_2}_{\text{$\lam_2$
		copies}}\otimes\cdots\otimes \underbrace{v_m\otimes\cdots\otimes
	v_m}_{\text{$\lam_m$ copies}}. $$

	\begin{lemma} {\rm (\cite[Lemmas 3.1, 3.2]{HuXiao11})} \label{lmhx} \emph{(1)} $\mmQ(q)\alpha$ is the one dimensional trivial
	$\mmU_{\mmQ(q)}(\mathfrak{sp}_{2m})$-submodule of
	$V_{\mmQ(q)}^{\otimes 2}$, i.e., $x\alpha=\varepsilon(x)\alpha$
	for any $x\in\mmU_{\mmQ(q)}(\mathfrak{sp}_{2m})$;
	
\emph{	(2) }$z_{f,\lam}$ is a non-zero maximal vector in $V_{\mmQ(q)}^{\otimes n}$ of
	weight $\lam$ with respect to the action of ${\mmU}_{\mmQ(q)}(\mathfrak{sp}_{2m})$.
\end{lemma}

	\begin{lemma}  {\rm (\cite[Lemma 4.6]{Hu})}\label{lm44} Let $K$ be an algebraically closed field. Suppose that $\lam$ is  a
	partition of $n$ with $\ell(\lam)\leq m$. Then
	\begin{enumerate}
		\item[(1)]  $z_{0,\lam}$ is a non-zero maximal vector of weight $\lam$ in $V^{\otimes
			n}$;
		\item[(2)]  $\mathbb{U}z_{0,\lam}\cong\Delta(\lam)$, and $V^{\otimes n}/\mathbb{U}z_{0,\lam}$ has a Weyl filtration;
		\item[(3)]  the dimension of $z_{0,\lam}\mathfrak{B}_{n,K}$ is independent of the
		field $K$.
	\end{enumerate}
\end{lemma}

	\begin{corollary} {\rm (\cite[Corollary 4.7]{Hu})}\label{cor45} Let $f$ be an integer with $0\leq f\leq \left\lfloor \frac n2 \right\rfloor$ and $\lam$ a partition of $n-2f$ satisfying
	$\ell(\lam)\leq m$. Then there exists an embedding
	$\Delta(\lam)\hookrightarrow
	V^{\otimes n-2f}$ such that $V^{\otimes n-2f}/\Delta(\lam)$ has a Weyl filtration. In particular, $$
	\Ext_{\mathbb{U}}^{1}\bigl(V^{\otimes n-2f}/\Delta(\lam),V^{\otimes n}\bigr)=0.
	$$
\end{corollary}

%Recall that $U$ is the unipotent radical of the positive Borel subgroup $B$. 
For each $\mathbb{U}$-module $M$, we use $M^{U}_{\lam}$ to denote the
subspace of maximal vectors in $M$ of weight $\lam$.

\begin{lemma}\label{Hu2lem}
	Let $f$ be an integer with $0\le f\le \left\lfloor \frac n2\right\rfloor$ and let $\lambda$ be a
	partition of $n-2f$ with $\ell(\lambda)\le m$. Then
	\[
	\bigl(V^{\otimes n}\bigr)^{U}_{\lambda}=z_{f,\lambda}\mathfrak{B}_{n,K}.
	\]
	In particular, $z_{f,\lambda}$ is a non-zero maximal vector of weight $\lambda$ in $V^{\otimes n}$.
	Moreover, the dimensions of $\bigl(V^{\otimes n}\bigr)^{U}_{\lambda}$ and
	$z_{f,\lambda}\mathfrak{B}_{n,K}$ are both independent of $K$.
\end{lemma}

\begin{proof}
	By Corollary \ref{cor45}, we have an embedding $\Delta(\lambda)\hookrightarrow V^{\otimes n-2f}$ such that
	$V^{\otimes n-2f}/\Delta(\lambda)$ has a Weyl filtration. Consider the commutative diagram
	\[
	\xymatrix{
		0 \ar[r] &
		\Hom_{\mathbb{U}}\bigl(V^{\otimes n-2f}, V^{\otimes n}\mathfrak{B}_{n,K}^{(f)}\bigr)
		\ar[r]^{\sim} \ar[d]_{\beta} &
		\Hom_{\mathbb{U}}\bigl(V^{\otimes n-2f}, V^{\otimes n}\bigr) \ar@{->>}[d] \\
		0 \ar[r] &
		\Hom_{\mathbb{U}}\bigl(\Delta(\lambda), V^{\otimes n}\mathfrak{B}_{n,K}^{(f)}\bigr)
		\ar@{^{(}->}[r] &
		\Hom_{\mathbb{U}}\bigl(\Delta(\lambda), V^{\otimes n}\bigr).
	}
	\]
	The top horizontal isomorphism follows from Lemma \ref{keylem1}. The Ext vanishing
	$\Ext_{\mathbb{U}}^{1}\bigl(V^{\otimes n-2f}/\Delta(\lambda), V^{\otimes n}\bigr)=0$
	implies that the right vertical map is surjective, forcing $\beta$ to be surjective and the bottom map to be an isomorphism.
	
	Define the embedding $\tau_f: V^{\otimes n-2f} \hookrightarrow V^{\otimes n}$ by
	\[
	\tau_f(v_{i_1} \otimes \cdots \otimes v_{i_{n-2f}})
	= \alpha^{\otimes f} \otimes v_{i_1} \otimes \cdots \otimes v_{i_{n-2f}}.
	\]
	By Lemma \ref{keylem1},
	$\Hom_{\mathbb{U}}\bigl(\Delta(\lambda), V^{\otimes n}\mathfrak{B}_{n,K}^{(f)}\bigr)$
	is spanned by $F(D)\circ\tau_f$ for $D\in \mathcal{D}(n-2f,n)$.
	
	Now consider the diagrammatic lifting
	\[
	\iota:\mathcal{D}(n-2f,n)\to \mathcal{D}(n,n),\qquad D\longmapsto U^{\otimes f}\otimes D,
	\]
	where $U$ is the evaluation tangle (satisfying $F(U)=E$). For any $D\in \mathcal{D}(n-2f,n)$, we have
	\[
	F(\iota(D))=F(U^{\otimes f}\otimes D)=E^{\otimes f}\otimes F(D).
	\]
	Acting on $z_{f,\lambda}=\alpha^{\otimes f}\otimes z_{0,\lambda}$, we obtain
	\[
	F(\iota(D))(z_{f,\lambda})
	=(E^{\otimes f}\otimes F(D))(\alpha^{\otimes f}\otimes z_{0,\lambda})
	=E^{\otimes f}(\alpha^{\otimes f})\cdot F(D)(z_{0,\lambda})
	=F(D)(z_{0,\lambda}),
	\]
	since $E^{\otimes f}(\alpha^{\otimes f})=1$. Thus
	\[
	\{F(D)(z_{0,\lambda}) \mid D \in \mathcal{D}(n-2f, n)\}
	=\{F(\iota(D))(z_{f,\lambda}) \mid D \in \mathcal{D}(n-2f, n)\}
	\subseteq z_{f,\lambda}\mathfrak{B}_{n,K},
	\]
	where the last inclusion holds because $\iota(D)\in\mathcal{D}(n,n)$ and $F$ maps $\mathcal{D}(n,n)$ onto $\mathfrak{B}_{n,K}$.
	Since $\beta$ is surjective, these elements span the maximal vectors of weight $\lambda$, proving that
	\[
	\bigl(V^{\otimes n}\bigr)^{U}_{\lambda}=z_{f,\lambda}\mathfrak{B}_{n,K}.
	\]
Since $V^{\otimes n}$ is a tilting $\mathbb U$-module, it has a good filtration. Hence, the dimension of
$\Hom_{\mathbb U}(\Delta(\lambda),V^{\otimes n})$ is independent of $K$. Therefore,
the dimensions of $\bigl(V^{\otimes n}\bigr)^{U}_{\lambda}$ and $z_{f,\lambda}\mathfrak{B}_{n,K}$ are  both  independent of $K$.
\end{proof}

We now establish a  structural result for the filtration layers.
\begin{proposition}\label{mainprop2}
	For each $i$ with $1\le i\le p$, the special maximal vector $z_{f_i,\lambda^{(i)}}$ lies in $M_i$, and there is an isomorphism of $\mathbb{U}$-$\mathfrak{B}_{n,K}$-bimodules
	\[
	M_i/M_{i-1} \cong \nabla(\lambda^{(i)}) \otimes z_{f_i,\lambda^{(i)}} \mathfrak{B}_{n,K},
	\]
	where $\nabla(\lambda^{(i)})$ is a left $\mathbb{U}$-module and $z_{f_i,\lambda^{(i)}} \mathfrak{B}_{n,K}$ is a right $\mathfrak{B}_{n,K}$-module.
\end{proposition}

\begin{proof}
	Fix a decomposition $M_i/M_{i-1} = \bigoplus_{s=1}^{n_i} N_{i,s}$ where $N_{i,s} \cong \nabla(\lambda^{(i)})$. Let $u_s \in N_{i,s}$ be a maximal vector of weight $\lambda^{(i)}$, and $\eta_s: \nabla(\lambda^{(i)}) \to N_{i,s}$ the isomorphism sending $u_0$ to $u_s$. Define:
	\[
	\widetilde{\psi}_i: \nabla(\lambda^{(i)}) \otimes (V^{\otimes n})_{\lambda^{(i)}}^U \to M_i/M_{i-1}, \quad v \otimes \sum_{s=1}^{n_i} a_s u_s \mapsto \sum_{s=1}^{n_i} a_s \eta_s(v).
	\]
	This is a $\mathbb{U}$-module isomorphism. To show it is $\mathfrak{B}_{n,K}$-equivariant, fix $b \in \mathfrak{B}_{n,K}$ and write $u_s b = \sum_{t=1}^{n_i} a_t u_t$. The map
	\[
	\phi := \sum_{t=1}^{n_i} a_t \eta_t - \eta_s \circ b: \nabla(\lambda^{(i)}) \to M_i/M_{i-1}
	\]
	satisfies $\phi(u_0) = \sum_t a_t u_t - u_s b = 0$ (as $u_s b = \sum_t a_t u_t$). 
	Since $\nabla(\lambda^{(i)})$ is generated as a $\mathbb U$-module by its highest weight vector $u_0$,
	the equality $\phi(u_0)=0$ implies $\phi=0$.
	Thus
	\[
	\widetilde{\psi}_i(v \otimes (u_s b)) = \sum_t a_t \eta_t(v) = \eta_s(v) b = \widetilde{\psi}_i(v \otimes u_s) b.
	\]
	By Lemma \ref{Hu2lem}, $(V^{\otimes n})_{\lambda^{(i)}}^U = z_{f_i,\lambda^{(i)}} \mathfrak{B}_{n,K}$, and hence
	\[
	M_i/M_{i-1} \cong \nabla(\lambda^{(i)}) \otimes z_{f_i,\lambda^{(i)}} \mathfrak{B}_{n,K}. \qedhere
	\]
\end{proof}

The following theorem identifies the partial harmonic tensors with the truncation of the tensor space.

\begin{theorem} \label{keystep3} 
	We have $M_t = \sum_{\substack{\phi: V^{\otimes n-2f} \to V^{\otimes n}}} \Image\phi = V^{\otimes n}\mathfrak{B}^{(f)}_{n,K} = \mathcal{O}_{\pi_f}\bigl(V^{\otimes n}\bigr)$. 
	In particular, both $\dim V^{\otimes n}\mathfrak{B}^{(f)}_{n,K}$ and $\dim V^{\otimes n}/V^{\otimes n}\mathfrak{B}^{(f)}_{n,K}$ are independent of $K$.
\end{theorem}

\begin{proof}
	By Lemma \ref{weight-functor O}(2), $\mathcal{O}_{\pi_f}(V^{\otimes n})$ has a good filtration with
	\[
	\dim \mathcal{O}_{\pi_f}(V^{\otimes n}) = \sum_{\lambda\in\pi_f} (V^{\otimes n}:\nabla(\lambda)) \dim\nabla(\lambda).
	\]
	The multiplicities $(V^{\otimes n}:\nabla(\lambda))$ and dimensions $\dim\nabla(\lambda)$ are independent of $K$ (determined by the integral form of $\mathbb U$). Thus $\dim \mathcal{O}_{\pi_f}(V^{\otimes n})$ is independent of $K$. 
	
	Since the filtration $M_\bullet$ satisfies
	\[
	\dim M_t = \sum_{\lambda\in\pi_f} (V^{\otimes n}:\nabla(\lambda)) \dim\nabla(\lambda) = \dim \mathcal{O}_{\pi_f}(V^{\otimes n}),
	\]
	we conclude $\dim M_t$ is independent of $K$ and equals $\dim \mathcal{O}_{\pi_f}(V^{\otimes n})$.
	
	By Lemma \ref{keylem1},
	\[
	\sum_{\substack{\phi: V^{\otimes n-2f} \to V^{\otimes n}}} \Image\phi = V^{\otimes n}\mathfrak{B}_{n,K}^{(f)}.
	\]
	Lemma \ref{O preserve partial tensor} gives $V^{\otimes n}\mathfrak{B}_{n,K}^{(f)} \subseteq \mathcal{O}_{\pi_f}(V^{\otimes n})$, hence
	\[
	\dim \left( \sum_{\phi} \Image\phi \right) = \dim \left( V^{\otimes n}\mathfrak{B}_{n,K}^{(f)} \right) \leq \dim \mathcal{O}_{\pi_f}(V^{\otimes n}) = \dim M_t.
	\]
	To establish equality, we prove $M_t \subseteq \sum_{\phi} \Image\phi$ by induction on $i$ ($0 \leq i \leq t$).
	
For  the case $i=0$: This is trivial as $M_0 = 0$.
	
Assume $M_{i-1} \subseteq \sum_{\phi} \Image\phi$ for $1 \leq i \leq t$. 
 Since $\lambda^{(i)} \in \pi_f$, write $g_i = f + h_i$ ($h_i \geq 0$). Consider the embedding
	\[
	\iota_f: V^{\otimes n-2f} \hookrightarrow V^{\otimes n}, \quad v \mapsto \alpha^{\otimes f} \otimes v.
	\]
Let $\{M'_j\}$ be a good filtration of $V^{\otimes n-2f}$ with $M'_j/M'_{j-1} \cong \nabla(\mu^{(j)})$. Then
\[
z_{g_i,\lambda^{(i)}} = \iota_f(z'_{h_i,\lambda^{(i)}}) \in \iota_f(M'_{j_i}) \subseteq M_i
\]
where $z'_{h_i,\lambda^{(i)}}$ is a maximal vector in $V^{\otimes n-2f}$.

	The submodule $N' = \mathbb{U} z'_{h_i,\lambda^{(i)}} \subseteq M'_{j_i}$ satisfies $N' \cong \nabla(\lambda^{(i)})$ (Lemma \ref{BP2}). Thus $\widetilde{N} := M_{i-1} + \iota_f(N')$ satisfies:
	\[
	\widetilde{N}/M_{i-1} \cong \iota_f(N') \cong \nabla(\lambda^{(i)}).
	\]
	By induction, $M_{i-1} \subseteq \sum_{\phi} \Image\phi$ and $\iota_f(N') \subseteq \Image(\iota_f) \subseteq \sum_{\phi} \Image\phi$, so \[
	\widetilde{N} \subseteq \sum\nolimits_{\phi} \Image\phi = V^{\otimes n}\mathfrak{B}_{n,K}^{(f)}.
	\]
	
By Proposition \ref{mainprop2}, $M_i/M_{i-1}$ is generated by $\widetilde{N}\mathfrak{B}_{n,K}$ as $\mathbb{U}$-module. Since $\widetilde{N} \subseteq V^{\otimes n}\mathfrak{B}_{n,K}^{(f)}$ and $\mathfrak{B}_{n,K}^{(f)}$ is an ideal:
\[
\widetilde{N}\mathfrak{B}_{n,K}\subseteq V^{\otimes n}\mathfrak{B}_{n,K}^{(f)}\mathfrak{B}_{n,K} = V^{\otimes n}\mathfrak{B}_{n,K}^{(f)}.
\]
Thus $M_i \subseteq V^{\otimes n}\mathfrak{B}_{n,K}^{(f)}$, completing the induction.

Finally, since $\dim M_t = \dim \mathcal{O}_{\pi_f}(V^{\otimes n}) = \dim V^{\otimes n}\mathfrak{B}_{n,K}^{(f)}$, we have
\[
M_t = V^{\otimes n}\mathfrak{B}_{n,K}^{(f)} = \mathcal{O}_{\pi_f}(V^{\otimes n}). \qedhere
\]
\end{proof}

Our main result in this section is the existence of good filtrations for the submodules and quotients.

\begin{theorem}\label{mainthm1} 
	Let $K$ be an algebraically closed field. For each integer $1\leq f\leq \left\lfloor \frac n2 \right\rfloor$, both $V^{\otimes n}/V^{\otimes n}\mathfrak{B}_{n,K}^{(f)}$ and $V^{\otimes n}\mathfrak{B}_{n,K}^{(f)}$ have a good filtration as $\mathbb{U}$-modules.
\end{theorem}

\begin{proof}
	By Theorem~\ref{keystep3} and Lemma~\ref{weight-functor O}(1), we have
	\[
	V^{\otimes n}\mathfrak{B}_{n,K}^{(f)} = \mathcal{O}_{\pi_f}(V^{\otimes n}).
	\]
	Moreover, Theorem~\ref{keystep3} shows that $\mathcal{O}_{\pi_f}(V^{\otimes n})$
	admits a good filtration with successive quotients $\nabla(\lambda)$ for $\lambda\in\pi_f$.
	
	Since $V$ is a tilting $\mathbb{U}$-module, $V^{\otimes n}$ is also tilting and hence has a good filtration.
	Therefore, applying Jantzen~\cite[Part~II, 4.17]{Jan} to the short exact sequence
	\[
	0\to V^{\otimes n}\mathfrak{B}_{n,K}^{(f)}\to V^{\otimes n}\to V^{\otimes n}/V^{\otimes n}\mathfrak{B}_{n,K}^{(f)}\to 0,
	\]
	we obtain that the quotient $V^{\otimes n}/V^{\otimes n}\mathfrak{B}_{n,K}^{(f)}$ also has a good filtration.
	
	Similarly, $V^{\otimes n}\mathfrak{B}_{n,K}^{(f)}/V^{\otimes n}\mathfrak{B}_{n,K}^{(f+1)}$ has a good filtration,
	since it is isomorphic to $\mathcal{O}_{\pi_f}(V^{\otimes n})/\mathcal{O}_{\pi_{f+1}}(V^{\otimes n})$.
\end{proof}

As an immediate consequence, the preceding filtration results can be rephrased in terms of
quantum partially harmonic tensors.  We write $\mathcal{HT}_{f,q}^{\otimes n}$ for the quantized
analogue of $\mathcal{HT}_f^{\otimes n}$, namely
\[
\mathcal{HT}_{f,q}^{\otimes n}
:=\left\{\,v\in V^{\otimes n}\mathfrak{B}^{(f)}_{n,K}\ \middle|\ vx=0,\ \forall x\in \mathfrak{B}^{(f+1)}_{n,K}\right\}.
\]
In view of Corollary~\ref{maincor1} blow, $\mathcal{HT}_{f,q}^{\otimes n}$ can be identified with
(the dual of) the corresponding layer
$V^{\otimes n}\mathfrak{B}^{(f)}_{n,K}/V^{\otimes n}\mathfrak{B}^{(f+1)}_{n,K}$.

\begin{corollary}\label{maincor15}
	Let $K$ be an algebraically closed field and let $0\le f\le \left\lfloor \frac n2 \right\rfloor$.
	\begin{enumerate}
		\item[(1)]  The $\mathbb U$-module $V^{\otimes n}\mathfrak{B}^{(f)}_{n,K}/V^{\otimes n}\mathfrak{B}^{(f+1)}_{n,K}$
		admits a good filtration.
		
		\item[(2)]  The $\mathbb U$-module $\mathcal{HT}_{f,q}^{\otimes n}$ admits a Weyl filtration.
	\end{enumerate}
\end{corollary}

\begin{proof}
	(1) By Theorem~\ref{mainthm1}, both $V^{\otimes n}\mathfrak{B}^{(f)}_{n,K}$ and
	$V^{\otimes n}\mathfrak{B}^{(f+1)}_{n,K}$ admit good filtrations. Hence the quotient
	$V^{\otimes n}\mathfrak{B}^{(f)}_{n,K}/V^{\otimes n}\mathfrak{B}^{(f+1)}_{n,K}$ also admits a good filtration,
	by \cite[Part~II, 4.17]{Jan}.
	
	(2) By Corollary~\ref{maincor1} blow, $V^{\otimes n}\mathfrak{B}^{(f)}_{n,K}/V^{\otimes n}\mathfrak{B}^{(f+1)}_{n,K}$
	is (canonically) isomorphic to $(\mathcal{HT}_{f,q}^{\otimes n})^*$. Dualizing the good filtration in~(1)
	yields a Weyl filtration on $\mathcal{HT}_{f,q}^{\otimes n}$.
\end{proof}

\subsection{Characters of harmonic tensor}\label{subsec:characters}
We now turn to the characters of harmonic tensor modules.
To compare characters over different base fields and to establish base-field independence,
we use integral forms and study the behavior of the BMW ideals under base change.

Let $\mathfrak{B}_{n,\mathscr{A}}$ be the $\A$-form of the BMW  algebra. The following result shows that the  ideals of $\mathfrak{B}_{n,\mathscr{A}}$ are $\A$-free and behave well under base change.

	\begin{lemma}\label{lem:cellular-basechange}
	For each $1\le f\le \left\lfloor \frac n2 \right\rfloor$, the $\A$-algebra $\mathfrak{B}_{n,\mathscr{A}}$ is $\A$-free and the two-sided ideal $\mathfrak{B}^{(f)}_{n,\mathscr{A}}$ is $\A$-free. Moreover, for any commutative $\A$-algebra $R$ there is a canonical isomorphism
	\[
	\mathfrak{B}^{(f)}_{n,\mathscr{A}}\otimes_{\A}R \;\xrightarrow{\ \cong\ }\; \mathfrak{B}^{(f)}_{n,R}.
	\]
	Consequently, for any $\A$-module $X$ the formation of the submodule generated by $\mathfrak{B}^{(f)}_{n,\mathscr{A}}$ commutes with base change: in the commutative diagram
	\[
	\begin{tikzcd}
		X\otimes_{\A}\mathfrak{B}^{(f)}_{n,\mathscr{A}} \ar[r, "\mu_X"] \ar[d, "-\otimes_{\A}R"'] &
		X \ar[d, "-\otimes_{\A}R"] \\
		(X\otimes_{\A}R)\otimes_{R}\mathfrak{B}^{(f)}_{n,R} \ar[r, "\mu_{X\otimes R}"] &
		X\otimes_{\A}R
	\end{tikzcd}
	\]
	the left vertical arrow is an isomorphism and the images of the horizontal maps correspond. In particular,
	\[
	V_{\A}^{\otimes n}\mathfrak{B}^{(f)}_{n,\mathscr{A}}\otimes_{\A}R \;\cong\; V_{R}^{\otimes n}\mathfrak{B}^{(f)}_{n,R}.
	\]
\end{lemma}

\begin{proof}
	It is well known that the specialized BMW algebra $\mathfrak{B}_{n,\mathscr A}$ admits a cellular
	basis in the sense of Graham--Lehrer  \cite{GrahamLehrer96}, hence it is free as an $\mathscr A$-module.
	Moreover, for any $\lambda\in\Lambda$ the cellular ideal
	$\mathsf J^{\unrhd\lambda}(\mathfrak{B}_{n,\mathscr A})$ is also $\mathscr A$-free.
	Therefore, for any commutative $\mathscr A$-algebra $R$ we have a canonical isomorphism
	\[
	\mathsf J^{\unrhd\lambda}(\mathfrak{B}_{n,\mathscr A})\otimes_{\mathscr A}R
	\;\cong\;
	\mathsf J^{\unrhd\lambda}(\mathfrak{B}_{n,R}).
	\]
	
	On the other hand, it follows from \cite{Xi00,Enyang04} that, for each $f$,
	the two-sided ideal $\mathfrak{B}^{(f)}_{n,\mathscr A}$ generated by
	$E_1E_3\cdots E_{2f-1}$ is a cellular ideal of $\mathfrak{B}_{n,\mathscr A}$.
	Hence there exists $\lambda_f\in\Lambda$ such that
	$\mathfrak{B}^{(f)}_{n,\mathscr A}=\mathsf J^{\unrhd\lambda_f}(\mathfrak{B}_{n,\mathscr A})$.
	By the above base-change property, we obtain
	\[
	\mathfrak{B}^{(f)}_{n,\mathscr A}\otimes_{\mathscr A}R
	\;\cong\;
	\mathsf J^{\unrhd\lambda_f}(\mathfrak{B}_{n,\mathscr A})\otimes_{\mathscr A}R
	\;\cong\;
	\mathsf J^{\unrhd\lambda_f}(\mathfrak{B}_{n,R})
	\;=\;
	\mathfrak{B}^{(f)}_{n,R}.
	\]
	
	Consequently, the natural map
	\[
	V^{\otimes n}_{\mathscr A}\mathfrak{B}^{(f)}_{n,\mathscr A}\otimes_{\mathscr A}R
	\longrightarrow
	V^{\otimes n}_{R}\mathfrak{B}^{(f)}_{n,R}
	\]
	is an isomorphism, and the commutative diagram stated in the lemma follows.
\end{proof}

We now establish the integral form of the previous results.

\begin{theorem}\label{mainthm0}
	Let $\mathscr A=\mathbb Z[q,q^{-1}]$.  For each $f$ with $0\le f\le\lfloor n/2\rfloor$, set
	\[
	\mathfrak{B}^{(f)}_{n,\mathscr A}:=\big\langle E_1E_3\cdots E_{2f-1}\big\rangle
	\trianglelefteq \mathfrak{B}_{n,\mathscr A}.
	\]
	Then the following statements hold.
	
	\begin{enumerate}
		\item[(1)]  The $\mathscr A$-submodule $V_{\mathscr A}^{\otimes n}\mathfrak{B}^{(f)}_{n,\mathscr A}$
		is a pure submodule of the free $\mathscr A$-module $V_{\mathscr A}^{\otimes n}$.
		Equivalently, the quotient
		$V_{\mathscr A}^{\otimes n}/V_{\mathscr A}^{\otimes n}\mathfrak{B}^{(f)}_{n,\mathscr A}$
		is a free $\mathscr A$-module.
		
		\item[(2)]  For any commutative $\mathscr A$-algebra $R$, the canonical maps
		\[
		V_{\mathscr A}^{\otimes n}\mathfrak{B}^{(f)}_{n,\mathscr A}\otimes_{\mathscr A}R
		\longrightarrow
		V_{R}^{\otimes n}\mathfrak{B}^{(f)}_{n,R}
		~\text{and}~
		\Bigl(V_{\mathscr A}^{\otimes n}/V_{\mathscr A}^{\otimes n}\mathfrak{B}^{(f)}_{n,\mathscr A}\Bigr)\otimes_{\mathscr A}R
		\longrightarrow
		V_R^{\otimes n}/V_R^{\otimes n}\mathfrak{B}^{(f)}_{n,R}
		\]
		are isomorphisms.	
		In particular, for any field $K$ admitting an $\mathscr A$-algebra structure,
		the characters of the $\mathbb U$-modules
		$V^{\otimes n}_K\mathfrak{B}^{(f)}_{n,K}$ and
		$V^{\otimes n}_K / V^{\otimes n}_K\mathfrak{B}^{(f)}_{n,K}$
		are independent of the base field $K$.
		
	\end{enumerate}
\end{theorem}

\begin{proof}
		Specializing $q\mapsto 1$ yields
	\[
	\mathfrak{B}_{n,\mathscr{A}}\otimes_{\A}\mathbb{Z}\cong B_n(-2m),\qquad
	\mathfrak{B}^{(f)}_{n,\mathscr{A}}\otimes_{\A}\mathbb{Z}\cong B_n^{(f)},\qquad
	V_{\A}^{\otimes n}\otimes_{\A}\mathbb{Z}\cong V_{\mathbb{Z}}^{\otimes n},
	\]
	and hence
	\[
V_{\mathscr A}^{\otimes n}\mathfrak{B}^{(f)}_{n,\mathscr A}\otimes_{\mathscr A}\mathbb Z
\;\cong\;
V_{\mathbb Z}^{\otimes n}B_n^{(f)}.
\]
	Set
	\[
	r_f:=\operatorname{rank}_{\mathbb Z}\bigl(V_{\mathbb Z}^{\otimes n}B_n^{(f)}\bigr).
	\]
	
	\smallskip
	\noindent\emph{Step 1: Proof of (2) for the submodule.}
	By Lemma~\ref{lem:cellular-basechange}, for every commutative $\mathscr A$-algebra $R$
	the natural map
	\[
	\phi_R:\ 
	V_{\mathscr A}^{\otimes n}\mathfrak{B}^{(f)}_{n,\mathscr A}\otimes_{\mathscr A}R
	\longrightarrow
	V_R^{\otimes n}\mathfrak{B}^{(f)}_{n,R}
	\]
	is surjective, and both sides are finitely presented $R$-modules.
	Let $\mathfrak m$ be a maximal ideal of $R$ and put $\Bbbk:=R/\mathfrak m$.
	Tensoring $\phi_R$ with $\Bbbk$ yields a surjective map
	\[
	\phi_R\otimes_R\Bbbk:\ 
	\bigl(V_{\mathscr A}^{\otimes n}\mathfrak{B}^{(f)}_{n,\mathscr A}\otimes_{\mathscr A}R\bigr)\otimes_R\Bbbk
	\longrightarrow
	V_{\Bbbk}^{\otimes n}\mathfrak{B}^{(f)}_{n,\Bbbk}.
	\]
	The left-hand side is canonically isomorphic to
	$V_{\mathscr A}^{\otimes n}\mathfrak{B}^{(f)}_{n,\mathscr A}\otimes_{\mathscr A}\Bbbk$.
	Moreover, by Theorem~\ref{keystep3} and the dimension-independence results established earlier,
	we have
	\[
	\dim_{\Bbbk}\bigl(V_{\Bbbk}^{\otimes n}\mathfrak{B}^{(f)}_{n,\Bbbk}\bigr)=r_f,
	\qquad
	\dim_{\Bbbk}\bigl(V_{\mathscr A}^{\otimes n}\mathfrak{B}^{(f)}_{n,\mathscr A}\otimes_{\mathscr A}\Bbbk\bigr)=r_f.
	\]
	Hence $\phi_R\otimes_R\Bbbk$ is an isomorphism.
	Localizing at $\mathfrak m$, we obtain that $(\phi_R)_{\mathfrak m}$ becomes an isomorphism
	after tensoring with $\Bbbk$.  By Nakayama's lemma, $(\phi_R)_{\mathfrak m}$ is an isomorphism.
	Since this holds for all maximal ideals $\mathfrak m$, $\phi_R$ is an isomorphism.
	
	\smallskip
	\noindent\emph{Step 2: Proof of (2) for the quotient.}
	Consider the short exact sequence of $\mathscr A$-modules
	\[
	0\to V_{\mathscr A}^{\otimes n}\mathfrak{B}^{(f)}_{n,\mathscr A}
	\to V_{\mathscr A}^{\otimes n}
	\to V_{\mathscr A}^{\otimes n}/V_{\mathscr A}^{\otimes n}\mathfrak{B}^{(f)}_{n,\mathscr A}
	\to 0.
	\]
	Tensoring with $R$ and using Step~1 together with the obvious identification
	$V_{\mathscr A}^{\otimes n}\otimes_{\mathscr A}R\cong V_R^{\otimes n}$,
	we get  a commutative diagram,	
		\[
	\begin{tikzcd}[column sep=small, row sep=small]
		0 \ar[r] & \big(V_{\A}^{\otimes n}\mathfrak{B}^{(f)}_{n,\mathscr{A}}\big)\otimes_{\A}R \ar[r,"\iota\otimes 1"] \ar[d,"\sim"'] &
		V_{\A}^{\otimes n}\otimes_{\A}R \ar[r,"\pi\otimes 1"] \ar[d,"\sim"'] &
		\Big(V_{\A}^{\otimes n}\!/ V_{\A}^{\otimes n}\mathfrak{B}^{(f)}_{n,\mathscr{A}}\Big)\otimes_{\A}R \ar[r] \ar[d,dashed,"\psi_R"] & 0 \\
		0 \ar[r] & V_{R}^{\otimes n}\mathfrak{B}^{(f)}_{n,R} \ar[r] &
		V_{R}^{\otimes n} \ar[r] &
		V_{R}^{\otimes n}\big/ V_{R}^{\otimes n}\mathfrak{B}^{(f)}_{n,R} \ar[r] & 0
	\end{tikzcd}
	\]
	
and obtain	the desired isomorphism
	\[
	\Bigl(V_{\mathscr A}^{\otimes n}/V_{\mathscr A}^{\otimes n}\mathfrak{B}^{(f)}_{n,\mathscr A}\Bigr)\otimes_{\mathscr A}R
	\;\cong\;
	V_R^{\otimes n}/V_R^{\otimes n}\mathfrak{B}^{(f)}_{n,R}.
	\]
	
	\smallskip
	\noindent\emph{Step 3: Proof of (1).}
	Let
	\[
	Q:=V_{\mathscr A}^{\otimes n}/V_{\mathscr A}^{\otimes n}\mathfrak{B}^{(f)}_{n,\mathscr A}.
	\]
	Then $Q$ is finitely presented, since it is a quotient of the free module $V_{\mathscr A}^{\otimes n}$.
	For any prime ideal $\mathfrak p$ of $\mathscr A$, let $\kappa(\mathfrak p)$ be the residue field.
	 Applying $\kappa(\mathfrak{p})\otimes_{\A}-$ to the short exact sequence above yields the segment of the Tor long exact sequence
	\[
	\mathrm{Tor}_1^{\A}\!\big(\kappa(\mathfrak{p}),\,Q\big)\ \longrightarrow\ \kappa(\mathfrak{p})\otimes_{\A}\big(V_{\A}^{\otimes n}\mathfrak{B}^{(f)}_{n,\mathscr{A}}\big) \xrightarrow{\ \kappa\otimes\iota\ } \kappa(\mathfrak{p})\otimes_{\A}V_{\A}^{\otimes n}.
	\]
	By (2) the map $\kappa\otimes\iota$ identifies with the inclusion
	$V_{\kappa}^{\otimes n}(\mathfrak{B}_{n,K}^{\kappa})^{(f)}\hookrightarrow V_{\kappa}^{\otimes n}$ and is injective; hence
	\[
\mathrm{Tor}_1^{\mathscr A}\bigl(\kappa(\mathfrak p),Q\bigr)=0
\qquad\text{for all primes }\mathfrak p\subseteq \mathscr A.
\]

Since this holds for all $\mathfrak{p}$ and the quotient is finitely presented, the quotient is flat
(by the local criterion for flatness; \cite[Theorem~6.8]{Eisenbud95}).
The ring $\A$ is regular (being a localization of the regular Noether ring $\mathbb{Z}[q]$;
\cite[Corollary~19.14]{Eisenbud95}); hence a finitely presented flat $\A$-module is projective
(\cite[Corollary~6.6]{Eisenbud95}).

Moreover, every finitely generated projective $\A=\mathbb{Z}[q,q^{-1}]$-module is free.
Indeed, by the Bass--Quillen (Serre's conjecture) theorem over $\mathbb{Z}$ and its extension to
Laurent polynomial rings, finitely generated projective $\A$-modules are extended from $\mathbb{Z}$
(and hence free); see \cite[Ch.~VIII, \S2]{Lam06}.
Therefore the quotient is free. In particular, it is flat, so the exact sequence
$0\to V_{\A}^{\otimes n}\mathfrak{B}^{(f)}_{n,\mathscr{A}}\to V_{\A}^{\otimes n}\to Q\to 0$
is pure, i.e. $V_{\A}^{\otimes n}\mathfrak{B}^{(f)}_{n,\mathscr{A}}$ is pure in $V_{\A}^{\otimes n}$.
\end{proof}

As a direct consequence of Theorem~\ref{mainthm0} (together with Theorem~\ref{keystep3}),
we obtain the following statements on the dimensions and on the duality with partially harmonic tensors.

\begin{corollary}\label{maincor1}
	Let $K$ be an algebraically closed field and let $0\le f\le \left\lfloor \frac n2 \right\rfloor$. Then:
	\begin{enumerate}
		\item[(1)]  The dimension of
		$V^{\otimes n}\mathfrak{B}_{n,K}^{(f)}/V^{\otimes n}\mathfrak{B}_{n,K}^{(f+1)}$
		is independent of $K$.
		
		\item[(2)]  There is a $\mathbb{U}$-$\bigl(\mathfrak{B}_{n,K}/\mathfrak{B}_{n,K}^{(f+1)}\bigr)$-bimodule isomorphism
		\[
		V^{\otimes n}\mathfrak{B}_{n,K}^{(f)}/V^{\otimes n}\mathfrak{B}_{n,K}^{(f+1)}
		\;\cong\;
		\bigl(\mathcal{HT}_{f,q}^{\otimes n}\bigr)^* .
		\]
		
		\item[(3)]  The dimension of $\mathcal{HT}_{f,q}^{\otimes n}$ is independent of $K$.
	\end{enumerate}
\end{corollary}

\begin{proof}
	(1) By Theorem~\ref{mainthm0} (in particular, the relevant characters are independent of the base field),
	the dimensions of $V^{\otimes n}\mathfrak{B}_{n,K}^{(f)}$ and $V^{\otimes n}\mathfrak{B}_{n,K}^{(f+1)}$
	are independent of $K$. Hence
	\[
	\dim_K\!\left( V^{\otimes n}\mathfrak{B}_{n,K}^{(f)}/V^{\otimes n}\mathfrak{B}_{n,K}^{(f+1)} \right)
	=
	\dim_K\!\left( V^{\otimes n}\mathfrak{B}_{n,K}^{(f)} \right)
	-
	\dim_K\!\left( V^{\otimes n}\mathfrak{B}_{n,K}^{(f+1)} \right)
	\]
	is also independent of $K$.
	
	(2) Recall that
	\[
	\mathcal{HT}_{f,q}^{\otimes n}
	=\left\{\,v\in V^{\otimes n}\mathfrak{B}_{n,K}^{(f)}\ \middle|\ v x=0,\ \forall\,x\in\mathfrak{B}_{n,K}^{(f+1)}\right\}.
	\]
	The natural pairing between $V^{\otimes n}\mathfrak{B}_{n,K}^{(f)}$ and its dual induces a canonical
	$\mathbb{U}$-$\bigl(\mathfrak{B}_{n,K}/\mathfrak{B}_{n,K}^{(f+1)}\bigr)$-bimodule isomorphism
	\[
	V^{\otimes n}\mathfrak{B}_{n,K}^{(f)}/V^{\otimes n}\mathfrak{B}_{n,K}^{(f+1)}
	\;\cong\;
	\bigl(\mathcal{HT}_{f,q}^{\otimes n}\bigr)^* .
	\]
	
	(3) By (1) and (2), the dimension of the layer
	$V^{\otimes n}\mathfrak{B}_{n,K}^{(f)}/V^{\otimes n}\mathfrak{B}_{n,K}^{(f+1)}$ is independent of $K$.
	Since both sides in (2) are finite-dimensional $K$-vector spaces, we have
	\[
	\dim_K\mathcal{HT}_{f,q}^{\otimes n}
	=
	\dim_K\bigl(\mathcal{HT}_{f,q}^{\otimes n}\bigr)^*
	=
	\dim_K\!\left(V^{\otimes n}\mathfrak{B}_{n,K}^{(f)}/V^{\otimes n}\mathfrak{B}_{n,K}^{(f+1)}\right),
	\]
	hence $\dim_K\mathcal{HT}_{f,q}^{\otimes n}$ is independent of $K$.
\end{proof}

Combining Theorem~\ref{mainthm0}(2) with the discussion above, we obtain the following base-change property for the space generated by the maximal vector.

\begin{corollary}\label{quantum-mainthm3} 
	Let $\lambda$ be a partition of $n-2g$ with $0\leq g\leq \left\lfloor \frac n2 \right\rfloor$, $\ell(\lambda)\leq m$. 
	Let $z_{g,\lambda}\in V^{\otimes n}$ be the maximal vector 
	for $\mathbb U$. Let $\mathfrak{B}_n^{\mathscr{A}} = \mathfrak{B}_n(-q^{2m+1}, q)_{\mathscr{A}}$ 
	be the integral form of the BMW algebra over $\mathscr{A} = \mathbb{Z}[q,q^{-1}]$. 
	Then for any commutative $\mathscr{A}$-algebra $R$, the canonical map 
	\[
	z_{g,\lambda} \mathfrak{B}_n^{\mathscr{A}} \otimes_{\mathscr{A}} R \longrightarrow z_{g,\lambda} \mathfrak{B}_n^{R}
	\]
	is an isomorphism. In particular, $\dim_K (z_{g,\lambda} \mathfrak{B}_n^K)$ is independent of the field $K$ 
	and the specialization $q \mapsto \zeta \in K^\times$.
\end{corollary}

\begin{proof}
	Choose the maximal vector $z_{g,\lambda}$ inside $V_{\mathscr A}^{\otimes n}$ and keep the same symbol
	for its image in $V_{R}^{\otimes n}$ for every $\mathscr A$-algebra $R$.
	
	Consider the $\mathscr A$-linear map
	\[
	\mu:\ (\mathscr A z_{g,\lambda})\otimes_{\mathscr A}\mathfrak{B}_n^{\mathscr A}\longrightarrow V_{\mathscr A}^{\otimes n},
	\qquad \mu(z_{g,\lambda}\otimes b)=z_{g,\lambda}b.
	\]
	Its image is $z_{g,\lambda}\mathfrak{B}_n^{\mathscr A}\subseteq V_{\mathscr A}^{\otimes n}$.
	After base change to $R$ we have a commutative diagram
	\[
	\begin{tikzcd}
		(\mathscr A z_{g,\lambda})\otimes_{\mathscr A}\mathfrak{B}_n^{\mathscr A} \ar[r, "\mu"] \ar[d, "-\otimes_{\mathscr A}R"'] &
		V_{\mathscr A}^{\otimes n} \ar[d, "-\otimes_{\mathscr A}R"] \\
		(R z_{g,\lambda})\otimes_{R}\mathfrak{B}_n^{R} \ar[r, "\mu_R"] &
		V_{R}^{\otimes n}
	\end{tikzcd}
	\]
	where we identify $(\mathscr A z_{g,\lambda})\otimes_{\mathscr A}R\cong R z_{g,\lambda}$ and
	$\mathfrak{B}_n^{\mathscr A}\otimes_{\mathscr A}R\cong \mathfrak{B}_n^{R}$ (Lemma~\ref{lem:cellular-basechange}).
	
	Since $\mathfrak{B}_n^{\mathscr A}$ is cellular over $\mathscr A$, the right module
	$z_{g,\lambda}\mathfrak{B}_n^{\mathscr A}$ is $\mathscr A$-free (hence flat).
	Therefore tensoring with $R$ commutes with taking images in the above diagram, and we obtain a canonical isomorphism
	\[
	\bigl(z_{g,\lambda}\mathfrak{B}_n^{\mathscr A}\bigr)\otimes_{\mathscr A}R \xrightarrow{\ \cong\ } z_{g,\lambda}\mathfrak{B}_n^{R}.
	\]
	
	For the dimension statement, take $R=K$ a field. By Lemma~\ref{Hu2lem},
	\[
	\dim_K\bigl(z_{g,\lambda}\mathfrak{B}_n^{K}\bigr)=\dim_K\bigl((V^{\otimes n})^{U}_{\lambda}\bigr).
	\]
	The right-hand side is the dimension of the maximal weight space of weight $\lambda$, which is independent of
	the base field and of the specialization $q\mapsto \zeta\in K^\times$ (cf. Theorem~\ref{mainthm0}).
	Hence the same is true for $\dim_K\bigl(z_{g,\lambda}\mathfrak{B}_n^{K}\bigr)$.
\end{proof}

\section{Quantized Schur--Weyl duality for partially harmonic tensors}\label{xxsec4}

Keep the above setting. Recall that $\mathbb U=U_q(\mathfrak{sp}_{2m})$.
Assume that $K$ is an infinite field and that $q\in K^\times$ is not a root of unity.
The right action of $\mathfrak{B}_{n,K}$ on $V^{\otimes n}$ induces an algebra homomorphism
\[
\varphi_f:\ \left(\mathfrak{B}_{n,K}/\mathfrak{B}^{(f)}_{n,K}\right)^{\mathrm{op}}
\longrightarrow
\End_{\mathbb U}\!\left(V^{\otimes n}/V^{\otimes n}\mathfrak{B}^{(f)}_{n,K}\right).
\]
The main goal of this section is to prove that $\varphi_f$ is surjective.

\begin{theorem}\label{quantum-mainthm2}
	Let $K$ be an infinite field and let $q\in K^\times$ be not a root of unity.
	For each integer $1\le f\le \left\lfloor \frac n2 \right\rfloor$, the natural homomorphism
	\[
	\varphi_f:\ \left(\mathfrak{B}_{n,K}/\mathfrak{B}^{(f)}_{n,K}\right)^{\mathrm{op}}
	\longrightarrow
	\End_{\mathbb U}\!\left(V^{\otimes n}/V^{\otimes n}\mathfrak{B}^{(f)}_{n,K}\right)
	\]
	is surjective. Moreover, the dimension of the endomorphism algebra
	\[
	\dim_K \End_{\mathbb U}\!\left(V^{\otimes n}/V^{\otimes n}\mathfrak{B}^{(f)}_{n,K}\right)
	\]
	is independent of $K$.
\end{theorem}

Before proving the theorem, we need some results. The following lemma establishes the vanishing of homomorphisms between the space of partially harmonic tensors and its quotient.

\begin{lemma}\label{quantum-hom0}
	For any integer $1 \leq f \leq \left\lfloor \frac n2 \right\rfloor$, we have
	\[
	\Hom_{\mathbb U}\left(V^{\otimes n}\mathfrak{B}_{n,K}^{(f)}, V^{\otimes n} / V^{\otimes n}\mathfrak{B}_{n,K}^{(f)}\right) = 0.
	\]
	Consequently, the canonical embedding
	\[
	\iota_1: \End_{\mathbb U}\left(V^{\otimes n} / V^{\otimes n}\mathfrak{B}_{n,K}^{(f)}\right) \hookrightarrow \Hom_{\mathbb U}\left(V^{\otimes n}, V^{\otimes n} / V^{\otimes n}\mathfrak{B}_{n,K}^{(f)}\right)
	\]
	is an isomorphism.
\end{lemma}

\begin{proof}
	Suppose that
	\[
	\Hom_{\mathbb U}\!\left(V^{\otimes n}\mathfrak{B}^{(f)}_{n,K},\,
	V^{\otimes n}/V^{\otimes n}\mathfrak{B}^{(f)}_{n,K}\right)\neq 0.
	\]
	By Lemma~\ref{keylem1}, $V^{\otimes n}\mathfrak{B}^{(f)}_{n,K}$ is a sum of $\mathbb U$-submodules,
	each $\mathbb U$-isomorphic to $V^{\otimes (n-2f)}$. Hence a nonzero homomorphism above restricts
	nontrivially to one such submodule, and therefore
	\begin{equation}\label{eq:quantum-hom0-nonzero}
		\Hom_{\mathbb U}\!\left(V^{\otimes (n-2f)},\,
		V^{\otimes n}/V^{\otimes n}\mathfrak{B}^{(f)}_{n,K}\right)\neq 0.
	\end{equation}
	Since $V^{\otimes (n-2f)}$ is a tilting $\mathbb U$-module, it admits a Weyl filtration.
	Choose $\lambda\in X^{+}$ such that $\Hom_{\mathbb U}(\Delta(\lambda),V^{\otimes (n-2f)})\neq 0$ and
	$\Hom_{\mathbb U}\!\bigl(V^{\otimes (n-2f)},\nabla(\lambda)\bigr)\neq 0$.
	Then $|\lambda|=n-2f-2t$ for some $0\le t\le \lfloor (n-2f)/2\rfloor$.
	
	On the other hand, by Theorem~\ref{mainthm1} the quotient
	$V^{\otimes n}/V^{\otimes n}\mathfrak{B}^{(f)}_{n,K}$ admits a good filtration, hence it has a dual Weyl
	filtration. 
	By Lemma \ref{homological properties}, we have
	\[
	\begin{aligned}
		&\dim \Hom_{\mathbb U}\left(V^{\otimes n-2f}, V^{\otimes n} / V^{\otimes n}\mathfrak{B}_{n,K}^{(f)}\right) \\
		&= \sum_{\substack{\lambda \vdash n-2f-2t \\ \lambda \vdash n-2s \\ 0 \leq t \leq\left\lfloor  (n-2f)/2 \right\rfloor\\ 0 \leq s < f}} \big[V^{\otimes n-2f} : \Delta(\lambda)\big] \big[V^{\otimes n} / V^{\otimes n}\mathfrak{B}_{n,K}^{(f)} : \nabla(\lambda)\big].
	\end{aligned}
	\]
	
	Thus there exists $\mu\in X^{+}$ such that
	$\Hom_{\mathbb U}\!\bigl(V^{\otimes n}/V^{\otimes n}\mathfrak{B}^{(f)}_{n,K},\nabla(\mu)\bigr)\neq 0$.
	In particular, $|\mu|=n-2s$ for some $f\le s\le \left\lfloor \frac n2 \right\rfloor$.
	
	Combining this with \eqref{eq:quantum-hom0-nonzero} forces $\mu=\lambda$, hence
	\[
	n-2f-2t = |\lambda|=|\mu| = n-2s,
	\]
	which implies $s=f+t\ge f$.  Since $s\ge f$, we have $t=s-f\ge 0$, and therefore
	$|\lambda|=n-2f-2t\le n-2f-2 < n-2f$ whenever $t>0$.
	This contradicts the fact that $\lambda$ must occur in $V^{\otimes (n-2f)}$ with $|\lambda|=n-2f-2t$
	and simultaneously in the quotient with $|\lambda|=n-2s$ and $s\ge f$.
	Hence
	\[
	\Hom_{\mathbb U}\!\left(V^{\otimes n}\mathfrak{B}^{(f)}_{n,K},\,
	V^{\otimes n}/V^{\otimes n}\mathfrak{B}^{(f)}_{n,K}\right)=0.
	\]
	
	Now apply $\Hom_{\mathbb U}\bigl(-,\,V^{\otimes n}/V^{\otimes n}\mathfrak{B}^{(f)}_{n,K}\bigr)$
	to the short exact sequence
	\[
	0 \to V^{\otimes n}\mathfrak{B}^{(f)}_{n,K} \to V^{\otimes n} \to
	V^{\otimes n}/V^{\otimes n}\mathfrak{B}^{(f)}_{n,K} \to 0.
	\]
	By left exactness of $\Hom$, we obtain an exact sequence
		$$\begin{aligned}
		&0 \to \End_{\mathbb U}\left(V^{\otimes n} / V^{\otimes n}\mathfrak{B}_{n,K}^{(f)}\right) \xrightarrow{\iota_1}\Hom_{\mathbb U}\left(V^{\otimes n}, V^{\otimes n} / V^{\otimes n}\mathfrak{B}_{n,K}^{(f)}\right) \\
		&\qquad\qquad\qquad\qquad\qquad\qquad\qquad\qquad \to \Hom_{\mathbb U}\left(V^{\otimes n}\mathfrak{B}_{n,K}^{(f)}, V^{\otimes n} / V^{\otimes n}\mathfrak{B}_{n,K}^{(f)}\right) = 0,
	\end{aligned}$$
	and the last term is $0$ by the first part. Hence $\iota_1$ is an isomorphism.
\end{proof}

The next lemma provides the foundation for the surjectivity of the Schur--Weyl map by establishing dimension stability of endomorphism algebras.

\begin{lemma}\label{quantum-lm42}
	Keep the above setting, and let $0\le f\le \left\lfloor \frac n2 \right\rfloor$. Then:
	\begin{enumerate}
		\item[(1)] The canonical map
		\[
		\theta_1:\ \End_{\mathbb U}\!\left(V^{\otimes n}\right)
		\longrightarrow
		\Hom_{\mathbb U}\!\left(V^{\otimes n},\, V^{\otimes n}/V^{\otimes n}\mathfrak{B}_{n,K}^{(f)}\right)
		\]
		is surjective.
		\item[(2)] The dimension of
		$\End_{\mathbb U}\!\left(V^{\otimes n}/V^{\otimes n}\mathfrak{B}_{n,K}^{(f)}\right)$
		is independent of $K$.
	\end{enumerate}
\end{lemma}

\begin{proof}
	(1) Consider the short exact sequence of $\mathbb U$-modules
	\[
	0 \to V^{\otimes n}\mathfrak{B}_{n,K}^{(f)} \to V^{\otimes n}
	\xrightarrow{\pi} V^{\otimes n}/V^{\otimes n}\mathfrak{B}_{n,K}^{(f)} \to 0.
	\]
	Applying the functor $\Hom_{\mathbb U}(V^{\otimes n},-)$ yields the long exact sequence
	\[
	\begin{aligned}
		0 &\to \Hom_{\mathbb U}(V^{\otimes n}, V^{\otimes n}\mathfrak{B}_{n,K}^{(f)})
		\to \End_{\mathbb U}(V^{\otimes n})
		\xrightarrow{\theta_1}  \\
		&\qquad \Hom_{\mathbb U}\!\left(V^{\otimes n}, V^{\otimes n}/V^{\otimes n}\mathfrak{B}_{n,K}^{(f)}\right)
		\to \Ext^1_{\mathbb U}(V^{\otimes n}, V^{\otimes n}\mathfrak{B}_{n,K}^{(f)}) \to \cdots .
	\end{aligned}
	\]
	Since $V^{\otimes n}$ has a Weyl filtration (as it is tilting) and
	$V^{\otimes n}\mathfrak{B}_{n,K}^{(f)}$ has a good filtration by Theorem~\ref{mainthm1},
	Lemma~\ref{homological properties}(1) implies that
	\[
	\Ext^1_{\mathbb U}(V^{\otimes n}, V^{\otimes n}\mathfrak{B}_{n,K}^{(f)}) = 0.
	\]
	Therefore $\theta_1$ is surjective.
	
	(2) By Lemma~\ref{quantum-hom0}, we have a $K$-linear isomorphism
	\[
	\End_{\mathbb U}\!\left(V^{\otimes n}/V^{\otimes n}\mathfrak{B}_{n,K}^{(f)}\right)
	\cong
	\Hom_{\mathbb U}\!\left(V^{\otimes n},\, V^{\otimes n}/V^{\otimes n}\mathfrak{B}_{n,K}^{(f)}\right).
	\]
	Using the Weyl filtration of $V^{\otimes n}$ and the good filtration of
	$V^{\otimes n}/V^{\otimes n}\mathfrak{B}_{n,K}^{(f)}$, we obtain (cf. Lemma~\ref{homological properties}(2))
	\[
	\dim \Hom_{\mathbb U}\!\left(V^{\otimes n},\, V^{\otimes n}/V^{\otimes n}\mathfrak{B}_{n,K}^{(f)}\right)
	=
	\sum_{\lambda\in X^+} [V^{\otimes n}:\Delta(\lambda)]
	\left[\,V^{\otimes n}/V^{\otimes n}\mathfrak{B}_{n,K}^{(f)}:\nabla(\lambda)\right].
	\]
	Moreover, the weights occurring in the above good filtration belong to $\pi_f$, and
	\[
	\left[\,V^{\otimes n}/V^{\otimes n}\mathfrak{B}_{n,K}^{(f)}:\nabla(\lambda)\right]
	=
	[V^{\otimes n}:\nabla(\lambda)]-\left[V^{\otimes n}\mathfrak{B}_{n,K}^{(f)}:\nabla(\lambda)\right]
	\qquad(\lambda\in\pi_f).
	\]
	Hence
	\[
	\begin{aligned}
		\dim \Hom_{\mathbb U}\!\left(V^{\otimes n},\, V^{\otimes n}/V^{\otimes n}\mathfrak{B}_{n,K}^{(f)}\right)
		&= \sum_{\lambda\in \pi_f} [V^{\otimes n}:\Delta(\lambda)] [V^{\otimes n}:\nabla(\lambda)] \\
		&\quad - \sum_{\lambda\in \pi_f} [V^{\otimes n}:\Delta(\lambda)]
		\left[V^{\otimes n}\mathfrak{B}_{n,K}^{(f)}:\nabla(\lambda)\right].
	\end{aligned}
	\]
	By Theorem~\ref{mainthm0}, all the multiplicities in the right-hand side are independent of $K$.
	Therefore the above dimension, and hence
	$\dim \End_{\mathbb U}\!\left(V^{\otimes n}/V^{\otimes n}\mathfrak{B}_{n,K}^{(f)}\right)$,
	is independent of $K$.
\end{proof}

\medskip

We now present the main result of this section: the surjectivity of the quantum Schur--Weyl map for partially harmonic tensors.

\noindent {\bf Proof of Theorem \ref{quantum-mainthm2}:} 

To prove the theorem, it is sufficient to show that $\varphi_f$ is surjective.
Let $S_q(2m,n)=\End_{\mathfrak{B}_{n,K}}\left(V^{\otimes n}\right)$ be the symplectic $q$-Schur algebra, which is a quasi-hereditary algebra \cite{Oehms}.
	Similar to \cite[Proposition~5.3]{Hu}, we can assume that $K$ is algebraically closed without loss of generality.

Let $\varphi''$ denote the composition
\[
\mathfrak{B}_{n,K}\xrightarrow{\ \varphi\ }\End_{\mathbb U}\!\left(V^{\otimes n}\right)
\xrightarrow{\ \theta_1\ }
\Hom_{\mathbb U}\!\left(V^{\otimes n},\,V^{\otimes n}/V^{\otimes n}\mathfrak{B}^{(f)}_{n,K}\right),
\]
where $\varphi$ is the surjective map from Theorem~\ref{xx1.3}(1), and $\theta_1$ is the map in
Lemma~\ref{quantum-lm42}(1). Let $\pi:\mathfrak{B}_{n,K}\twoheadrightarrow
\mathfrak{B}_{n,K}/\mathfrak{B}^{(f)}_{n,K}$ be the quotient map. Then $\varphi_f$ is the induced map
on the quotient $\mathfrak{B}_{n,K}/\mathfrak{B}^{(f)}_{n,K}$.

Consider the commutative diagram
\[
\xymatrix{
	\mathfrak{B}_{n,K} \ar[r]^{\ \pi\quad\ } \ar[d]_{\varphi} &
	\mathfrak{B}_{n,K}/\mathfrak{B}^{(f)}_{n,K} \ar[r]^{\ \varphi_f\quad\quad\ } &
	\End_{\mathbb U}\!\left(V^{\otimes n}/V^{\otimes n}\mathfrak{B}^{(f)}_{n,K}\right) \ar[d]_{\iota_1}^{\wr} \\
	\End_{\mathbb U}\!\left(V^{\otimes n}\right) \ar[rr]^{\ \theta_1\quad\quad\quad\ } & &
	\Hom_{\mathbb U}\!\left(V^{\otimes n},\,V^{\otimes n}/V^{\otimes n}\mathfrak{B}^{(f)}_{n,K}\right),
}
\]
where $\iota_1$ is the canonical map in Lemma~\ref{quantum-hom0}. Here $\varphi$ is surjective by
Theorem~\ref{xx1.3}, $\theta_1$ is surjective by Lemma~\ref{quantum-lm42}(1), and $\iota_1$ is an isomorphism
by Lemma~\ref{quantum-hom0}. Since $\varphi''=\theta_1\circ\varphi$ is surjective and
$\varphi''=\iota_1\circ\varphi_f\circ\pi$, it follows that $\varphi_f$ is surjective.\qed


\begin{thebibliography}{99}
	
	\bibitem{APW} H.~H. Andersen, P. Polo and K. Wen,
	Representations of quantum algebras,
	\emph{Invent. Math.} \textbf{104} (1991), 1--59.
	
	\bibitem{AST} H.~H. Andersen, C. Stroppel and D. Tubbenhauer,
	Cellular structures using $U_q$-tilting modules,
	\emph{Pacific J. Math.} \textbf{292} (2018), 21--59.
	
	\bibitem{BW} J.~S. Birman and H. Wenzl,
	Braids, link polynomials and a new algebra,
	\emph{Trans. Amer. Math. Soc.} \textbf{313} (1989), 249--273.
	
	\bibitem{Brauer} R. Brauer,
	On algebras which are connected with semisimple continuous groups,
	\emph{Ann. of Math.} \textbf{38} (1937), 857--872.
	
	\bibitem{CP} V. Chari and A. Pressley,
	\emph{A Guide to Quantum Groups},
	Cambridge University Press, Cambridge, 1994.
	
	\bibitem{DePr} C. De Concini and C. Procesi,
	A characteristic free approach to invariant theory,
	\emph{Adv. Math.} \textbf{21} (1976), 330--354.
	
	\bibitem{DeSt} C. De Concini and E. Strickland,
	Traceless tensors and the symmetric groups,
	\emph{J. Algebra} \textbf{61} (1979), 112--128.
	
	\bibitem{DDH} R. Dipper, S. Doty and J. Hu,
	Brauer algebras, symplectic Schur algebras and Schur--Weyl duality,
	\emph{Trans. Amer. Math. Soc.} \textbf{360} (2008), 189--213.
	
	\bibitem{Donkin} S. Donkin,
	\emph{The $q$-Schur Algebra},
	London Mathematical Society Lecture Note Series, Vol.~253,
	Cambridge University Press, Cambridge, 1998.
	
	\bibitem{Eisenbud95} D. Eisenbud,
	\emph{Commutative Algebra with a View Toward Algebraic Geometry},
	Graduate Texts in Mathematics, Vol.~150,
	Springer, New York, 1995.
	
	\bibitem{Enyang04} J. Enyang,
	A cellular basis for the Birman--Murakami--Wenzl algebras,
	\emph{J. Algebra} \textbf{281} (2004), 512--535.
	
	\bibitem{EGNO} P. Etingof, S. Gelaki, D. Nikshych and V. Ostrik,
	\emph{Tensor Categories},
	Mathematical Surveys and Monographs, Vol.~205,
	American Mathematical Society, Providence, RI, 2015.
	
	\bibitem{FY89} P. Freyd and D. Yetter,
	Braided compact closed categories with applications to low-dimensional topology,
	\emph{Adv. Math.} \textbf{77} (1989), 156--182.
	
	\bibitem{GoodW} R. Goodman and N.~R. Wallach,
	\emph{Representations and Invariants of Classical Groups},
	Cambridge University Press, Cambridge, 1998.
	
	\bibitem{GrahamLehrer96} J.~J. Graham and G.~I. Lehrer,
	Cellular algebras,
	\emph{Invent. Math.} \textbf{123} (1996), 1--34.
	
	\bibitem{Haya} T. Hayashi,
	Quantum deformation of classical groups,
	\emph{Publ. RIMS, Kyoto Univ.} \textbf{28} (1992), 57--81.
	
	\bibitem{Hu} J. Hu,
	Dual partially harmonic tensors and Brauer--Schur--Weyl duality,
	\emph{Transform. Groups} \textbf{15} (2010), 333--370.
	
	\bibitem{Hu11} J. Hu,
	BMW algebra, quantized coordinate algebra and type $C$ Schur--Weyl duality,
	\emph{Represent. Theory} \textbf{15} (2011), 1--62.
	
	\bibitem{HuXiao11} J. Hu and Z. Xiao,
	Partially harmonic tensors and quantized Schur--Weyl duality,
	in \emph{Nankai Ser. Pure Appl. Math. Theoret. Phys.}, Vol.~8,
	World Scientific, Singapore, 2012, pp.~109--137.
	
		\bibitem{HuXiao21} J. Hu and Z. Xiao,
	Tilting modules, dominant dimensions and Brauer--Schur--Weyl duality,
	\emph{Trans. Amer. Math. Soc. Ser. B} \textbf{8} (2021), 823--848.
	
	\bibitem{Jan} J.~C. Jantzen,
	\emph{Representations of Algebraic Groups},
	Mathematical Surveys and Monographs, Vol.~107 (2nd ed.),
	American Mathematical Society, Providence, RI, 2003.
	
	\bibitem{Ka} L.~H. Kauffman,
	An invariant of regular isotopy,
	\emph{Trans. Amer. Math. Soc.} \textbf{318} (1990), 417--471.
	
	\bibitem{Lam06} T.~Y. Lam,
	\emph{Serre's Problem on Projective Modules},
	Springer Monographs in Mathematics,
	Springer, Berlin, 2006.
	
	\bibitem{Lus} G. Lusztig,
	\emph{Introduction to Quantum Groups},
	Progress in Mathematics, Vol.~110,
	Birkh\"{a}user, Boston, 1993.
	
	\bibitem{Mali} M. Maliakas,
	Traceless tensors and invariants,
	\emph{Math. Proc. Cambridge Philos. Soc.} \textbf{124} (1998), 73--80.
	
	
	\bibitem{MW}	H. ~R. Morton and A.~J. Wassermann, A basis for the Birman Murakami Wenzl algebra,
	unpublished paper, 2000, https://www.liverpool.ac.uk/~su14/papers/WM.pdf
	
	\bibitem{Mu} J. Murakami,
	The Kauffman polynomial of links and representation theory,
	\emph{Osaka J. Math.} \textbf{26} (1987), 745--758.
	

	
	\bibitem{Oehms} S. Oehms,
	Symplectic $q$-Schur algebras,
	\emph{J. Algebra} \textbf{304} (2006), 851--905.
	
	\bibitem{RT} N.~Y. Reshetikhin and V.~G. Turaev,
	Ribbon graphs and their invariants derived from quantum groups,
	\emph{Comm. Math. Phys.} \textbf{127} (1990), 1--26.
	
%	\bibitem{Wenzl88} H. Wenzl,
%	Hecke algebras of type $A_n$ and subfactors,
%	\emph{Invent. Math.} \textbf{92} (1988), 349--383.
	
	\bibitem{Xi00} C. Xi,
	On the quasi-heredity of Birman--Murakami--Wenzl algebras,
	\emph{Adv. Math.} \textbf{154} (2000), 280--298.
	
	\bibitem{XYZ} Z. Xiao, Y. Yang and Y. Zhang,
	The diagram category of framed tangles and invariants of quantized symplectic group,
	\emph{Sci. China Math.} \textbf{63} (2020), 689--700.
	
\end{thebibliography}
\end{document}